\documentclass[11pt, notitlepage]{article}
\usepackage{amssymb,amsmath,comment}
\catcode`\@=11 \@addtoreset{equation}{section}
\def\thesection{\arabic{section}}

\def\theequation{\thesection.\arabic{equation}}
\catcode`\@=12
\usepackage{colortbl}
\usepackage{a4wide}
\usepackage{parskip}
\usepackage{xcolor}

\newcommand{\ds} {\displaystyle}
\newcommand{\e}{\epsilon}
\newcommand{\pa} {\partial}
\newcommand{\al} {\alpha}
\newcommand{\ba} {\beta}
\newcommand{\de} {\delta}

\newcommand{\Om} {\Omega}
\newcommand{\ra} {\rightarrow}
\newcommand{\rp} {\rightharpoonup}
\newcommand{\De} {\Delta}
\newcommand{\la} {\lambda}

\newcommand{\noi} {\noindent}
\newcommand{\na} {\nabla}

\newcommand{\mc} {\mathcal}

\newcommand{\ld} {\langle}
\newcommand{\rd} {\rangle}

\usepackage[all]{xy}
\catcode`\@=11
\def\theequation{\@arabic{\c@section}.\@arabic{\c@equation}}
\catcode`\@=12

\def\QED{\hfill {$\square$}\goodbreak \medskip}

\newtheorem{Theorem}{Theorem}[section]
\newtheorem{Lemma}[Theorem]{Lemma}
\newtheorem{Proposition}[Theorem]{Proposition}

\newtheorem{Remark}[Theorem]{Remark}
\newtheorem{Definition}[Theorem]{Definition}

\begin{document}
\vspace{0.01in}

\title
{Critical growth  elliptic problems  involving
	 Hardy-Littlewood-Sobolev critical exponent  in non-contractible domains }

\author{ {\bf Divya Goel\footnote{e-mail: divyagoel2511@gmail.com} \; 
	and \;  K. Sreenadh\footnote{
		e-mail: sreenadh@maths.iitd.ac.in}} \\ Department of Mathematics,\\ Indian Institute of Technology Delhi,\\
	Hauz Khaz, New Delhi-110016, India. }

\date{}

\maketitle

\begin{abstract}
	\noi The paper is concerned with the existence and multiplicity of positive solutions of the nonhomogeneous Choquard equation over an annular type bounded domain. Precisely, we consider the following equation
	\[
	-\De u = \left(\int_{\Om}\frac{|u(y)|^{2^*_{\mu}}}{|x-y|^{\mu}}dy\right)|u|^{2^*_{\mu}-2}u+f  \; \text{in}\;
	\Om,\quad 
	u = 0 \;    \text{ on } \pa \Om ,
	\]
	where $\Om$ is a smooth  bounded annular domain in $\mathbb{R}^N( N\geq 3)$,  $2^*_{\mu}=\frac{2N-\mu}{N-2}$, $f \in L^{\infty}(\Om)$ and $f \geq 0$.  We prove the  existence of  four positive solutions of the above problem  using the  Lusternik-Schnirelmann theory and varitaional methods,  when the inner hole  of the annulus is sufficiently small.
	
	\medskip
	
	\noi \textbf{Key words:} Hardy-Littlewood-Sobolev inequality, critical problems, non-contractible domains.
	
	\medskip
	
	\noi \textit{2010 Mathematics Subject Classification: 35A15, 35J60, 35J20.} 
	
\end{abstract}
\newpage
\section{Introduction}
In the pioneering work,  Tarantello \cite{taren} studied the nonhomogeneous elliptic equation
\begin{equation}\label{nh1}
-\De u  = |u|^{2^*-2}u +f  \; \text{ in }\;
\Om,\;\;
 u  = 0 \;    \text{ on } \pa \Om ,
\end{equation} 
   where $2^*=\frac{2N}{N-2}$ is the critical Sobolev exponent and $\Om$ is a bounded domain in $\mathbb{R}^N$ with smooth boundary.  If $f \in H^{-1}$ then it is shown that there exists at least two solutions of \eqref{nh1} by using  variational methods.  Min and song \cite{daohuan} proved the existence of two positive solutions of the following nonhomogeneous elliptic equation
\begin{equation}\label{nh23}
\begin{aligned}
-\De u  = f(x,u(x))+h  \; \text{ in }\;
\mathbb{R}^N
\end{aligned}
\end{equation} 
where $f(x,u)$ is a Carath\'{e}odory function with subcritical grotwh at $\infty.$
Further, many researchers investigated \eqref{nh1} and \eqref{nh23} for the existence and multiplicity of solutions. For  details, we refer  \cite{clapp2, clapp3,he, hirano2,wu1} and references therein.  
Recently, Gao and Yang  \cite{zampyang} proved the existence of two positive solutions of the  nonhomogeneous Choquard  equation involving Hardy-Littlewood-Sobolev critical exponent using the splitting Nehari manifold method of Tarantello \cite{taren}.

%

The existence, uniqueness, and multiplicity of positive solutions of the nonlocal elliptic equation, precisely the Choquard equation both for mathematical analysis and in perspective of physical models has recently  gained significant  attention amongst researchers. As an instance, in 1954  Pekar \cite{pekar} proposed the equation
\begin{equation}\label{nh14}
-\De u +u = \left(\frac{1}{|x|}* |u|^2\right) u \text{ in } \mathbb{R}^3
\end{equation}
to study the quantum theory of polaron. Later in 1976, Ph. Choquard \cite{ehleib} examined the steady state of one component plasma approximation in Hartee-Fock theory using \eqref{nh14}. In \cite{ehleib}, Leib proved the existence and uniqueness of the ground state of \eqref{nh14}.  The work of Moroz and Schaftingen enriches the literature of Choquard equations. In  \cite{moroz2}  authors  studied  the following Choquard equation 
\begin{equation}\label{nh15}
-\De u + V u
=\left(\ds I_{\al}* F(u)\right)F^\prime(u),\;\;  \text{in } \mathbb{R}^N, 
\end{equation}
where $\al \in (0,N),\;  N\geq 3,\; I_\al$ is the Riesz Potential  and  $F(u) \in C^1(\mathbb{R}, \; \mathbb{R})$ with sub critical growth. In this work authors established the existence of ground state soloutions of \eqref{nh15} and 
assuming some suitable growth conditions on  $F$ and $V$, they  studied the properties like constant sign solutions and radial symmetry of the solution. Moreover, authors  proved the Poho\v zaev identity and nonlocal Brezis-Kato type estimate. Interested readers are referred to \cite{nodal,moroz1,moroz4,moroz3}  and references therein for the study of Choquard equation on the unbounded domain.

%
\noi  Concerning the boundary value problems of Choquard equation,  Gao and  Yang \cite{yang}   studied  the  Brezis-Nirenberg type existence results for the following  critical equation
 \begin{equation}\label{nh16}
 -\De u = \la {h(u)}+ \left(\int_{\Om}\frac{|u(y)|^{2^*_{\mu}}}{|x-y|^{\mu}}dy\right)|u|^{2^*_{\mu}-2}u \text{ in } \Om, \quad
 u=0 \text{ on } \pa \Om, 
 \end{equation}
\noi where $\la>0$, $0<\mu<N$, $h(u)=u$, $\Om$ is a smooth bounded domain in $\mathbb{R}^N$. Later  in \cite{yangjmaa}  authors  proved the existence and multiplicity  of positive solutions for convex and convex-concave type nonlinearities ($h(u)=u^q, 0<q<1$) using variational methods.

 The geometry of the domain $\Om$  plays an essential and significant role on the existence and multiplicity of the elliptic boundary value problems. Indeed, in \cite{coron}, Coron proved the existence of a high energy positive solution of the problem
 \begin{align}\label{nh26}
 -\De u=|u|^{2^*-2}u   \; \text{in}\;
 \Om,\quad 
 u = 0 \;    \text{ on } \pa \Om ,
 \end{align}
 where $\Om$ is a bounded domain in $\mathbb{R}^N(N \geq 3),   $
 precisely an annulus  with a small hole.
Later in \cite{bahri},  Bahri and Coron,  proved that  a positive solution always exists as long as the domain has non-trivial homology with $\mathbb{Z}_2$-coefficients.  In \cite{benci2}, Benci and Cerami studied the following equation
\begin{equation}\label{nh18}
 -\varepsilon \De u +u =f(u)   \; \text{in}\;
\Om,\quad 
u = 0 \;    \text{ on } \pa \Om ,
\end{equation}
where $\varepsilon \in \mathbb{R}^+ ,\;  \Om$ is a bounded domain in $\mathbb{R}^N(N \geq 3)$ and $f:\mathbb{R}_+ \ra \mathbb{R} $ is a $C^{1,1}$ function. Here authors proved that there exists  $\e^*>0$ such that for all $\varepsilon \in (0, \e^*)$, \eqref{nh18} has   $\text{cat}(\Om)+1$ solutions under some growth conditions on the function $f$. Since then, the study of existence and multiplicity of solutions of  elliptic equations over non-contractible domain has  been substantially studied, for instance, \cite{benci1,benci,dancer,he, rey, sign} and references therein. The existence of high energy solution of \eqref{nh26} is a much  more delicate issue. In this spirit, recently in \cite{choqcoron} Goel, R\u adulescu and Sreenadh studied the Coron problem for Choquard equations. Here authors proved the existence of a positive  high energy solution for the problem $(P_f)$ when $f(x)\equiv 0$ and  
  $\Om$ is a smooth  bounded domain in $\mathbb{R}^N( N\geq 3)$ 
  satisfying the following condition
\begin{itemize}
	\item [(A)] There exists constants $0<R_1<R_2<\infty$ such that 
	\begin{align*}
	\{ x \in \mathbb{R}^N \;:\; R_1<|x|<R_2 \} \subset\Om, \qquad
	\{ x \in \mathbb{R}^N \;:\; |x|<R_1 \} \nsubseteq \overline{\Om}.
	\end{align*}
\end{itemize} 
 
  In the light of above works, in this article, we study  following problem
 \begin{equation*}
 (P_f)\;
 \left\{
 -\De u
 =\left(\ds \int_{\Om}\frac{|u^+(y)|^{2^*_{\mu}}}{|x-y|^{\mu}}dy\right)|u^+|^{2^*_{\mu}-2}u^++f,\;\;  \text{in } \Omega,\;\;
 u=0 \; \text{ on } \pa \Om,
 \right.
 \end{equation*}
 \noi where $2^*_{\mu}=\frac{2N-\mu}{N-2}$, is the critical exponent in the sense of Hardy-Littlewood-Sobolev inequality \eqref{co9} and $ f \in \hat{F}$  with $\hat{F}:= \{f\; : \; f \in  L^{\infty}(\Om),\; f \geq 0,\;f\not \equiv 0 \}$.  
 The domain $\Om \subset \mathbb{R}^N( N\geq 3)$ satisfies the condition $(A)$. Here we prove the  existence of four solutions of the problem $(P_f)$.  To achieve this,  we first seek the help of Nehari manifold associated with $(P_f)$ to prove the existence of the first solution  (say $u_1$). To proceed further, we  prove many new estimates on the convolution terms involving the minimizers of  best constant  $S_{H,L}$ (see Lemma \ref{nhlem8}, \ref{nhlem9} and  \ref{nhlem33} ). With the help of these estimates we prove that the minima of the functional over $\mc{N}_f$ is  below the first critical level where the first critical level is 
  \begin{align*}
  \mathcal{J}_f(u_1)+	\frac{N-\mu+2}{2(2N-\mu)}S_{H,L}^{\frac{2N-\mu}{N-\mu+2}}.
  \end{align*}
   Here $\mathcal{J}_f$ is the energy functional associated to $(P_f)$ (defined in \eqref{nh25}). Moreover, $\mathcal{J}_f$ satisfies the Palais-Smale  condition below the first critical level. Subsequently, we show the existence of the second and the third solution of  $(P_f)$, in $\mathcal{N}_f^-$ (a closed subset of the Nehari manifold)  by using a well-known result of Ambrosetti \cite{ambrosetti}(see Lemma \ref{nhlem31}) and assumption $(A)$.  To study the existence of the fourth solution which is a high energy solution, we  prove that the functional $\mathcal{J}_f$ satisfies  the Palais-Smale  condition between the   first and  the second critical levels, where the second critical level is
  \begin{align*}
\inf_{u \in \mathcal{N}_f^-} \mathcal{J}_f(u)+	\frac{N-\mu+2}{2(2N-\mu)}S_{H,L}^{\frac{2N-\mu}{N-\mu+2}}.
  \end{align*}
  To prove the existence of fourth solution, we use the minmax Lemma (See Lemma \ref{nhlem25}).
To the best of our knowledge, there is no work on the existence of four solutions of Choquard equations $(P_f)$ in  non-contractible domains. With this introduction, we  state our main result.
\begin{Theorem}\label{nhthm1}
	Assume  $\mu <\min\{4,N\}$, $f \in L^{\infty}(\Om)$ and $f \geq 0$ and   $\Om$  be a bounded domain satisfying the conditon (A).
Then  there exists $e^*>0$ such that   $(P_f)$ has at least three positive solutions whenever  $0<\|f\|_{H^{-1}}<e^*$. Moreover, if $R_1$ is small enough then there exists  $e^{**}>0$ such that   $(P_f)$ has at least four positive solutions whenever $0<\|f\|_{H^{-1}}<e^{**}$.
\end{Theorem}
The paper is organized as follows: In Section 2, we give the variational framework and preliminary results. In section 3, using the Nehari manifold technique, we prove the existence of the first solution. In section 4, we prove some crucial estimates of the minimizer of $S_{H, L}$(defined in \eqref{nh5}) and analyze the  Palais-Smale sequences. In section 5, we prove the existence of the second and third solution. In section 6, we prove the existence of the fourth solution.
\section{Variational framework and preliminary results}
 We start with the familiar Hardy-Littlewood-Sobolev Inequality which leads to the study of nonlocal Choquard equation using  variational methods.
\begin{Proposition}\cite{leib}(\textbf{Hardy-Littlewood-Sobolev Inequality})
	Let $t,r>1$ and $0<\mu <N$ with $1/t+\mu/N+1/r=2$, $f\in L^t(\mathbb{R}^N)$ and $h\in L^r(\mathbb{R}^N)$. There exists a sharp constant $C(t,r,\mu,N)$ independent of $f$ and $h$ such that 
	\begin{equation}\label{co9}
	\int_{\mathbb{R}^N}\int_{\mathbb{R}^N}\frac{f(x)h(y)}{|x-y|^{\mu}}~dxdy \leq C(t,r,\mu,N) \|f\|_{L^t(\mathbb{R}^N)}\|h\|_{L^r(\mathbb{R}^N)}.
	\end{equation}
	If $t=r=2N/(2N-\mu)$, then 
	\begin{align*}
	C(t,r,\mu,N)=C(N,\mu)= \pi^{\frac{\mu}{2}}\frac{\Gamma(\frac{N}{2}-\frac{\mu}{2})}{\Gamma(N-\frac{\mu}{2})}\left\lbrace \frac{\Gamma(\frac{N}{2})}{\Gamma(\frac{\mu}{2})}\right\rbrace^{-1+\frac{\mu}{N}}.
	\end{align*}
	Equality holds in  \eqref{co9} if and only if $f\equiv (constant)h$ and 
	\begin{align*}
	h(x)= A(\gamma^2+|x-a|^2)^{(2N-\mu)/2}
	\end{align*} 
	for some $A\in \mathbb{C}, 0\neq \gamma \in \mathbb{R}$ and $a \in \mathbb{R}^N$. \QED
\end{Proposition}
\noi The best constant for the embedding $D^{1,2}(\mathbb{R}^N)$ into  $L^{2^*}(\mathbb{R}^N)$ (where  $2^*= \frac{2N}{N-2}$ )is defined as 
\[
S=\inf_{ u \in D^{1,2}(\mathbb{R}^N) \setminus \{0\}} \left\{ \int_{\mathbb{R}^N} |\na u|^2 dx:\; \int_{\mathbb{R}^N}|u|^{2^{*}}dx=1\right\}.\]
Consequently, we define 
\begin{equation}\label{nh5}
S_{H,L}= \inf_{ u \in D^{1,2}(\mathbb{R}^N)\setminus \{0\}} \left\{  \int_{\mathbb{R}^N}|\nabla u|^2 dx:\;  \int_{\mathbb{R}^{N}} \int_{\mathbb{R}^N } \frac{|u(x)|^{2^{*}_{\mu}}|u(y)|^{2^{*}_{\mu}}}{|x-y|^{\mu}}~dxdy=1 \right\}
\end{equation}

\begin{Lemma}\label{nhlem7}
	\cite{yang}
	The constant $S_{H,L}$ defined in \eqref{nh5} is achieved if and only if 
	\begin{align*}
	u=C\left(\frac{b}{b^2+|x-a|^2}\right)^{\frac{N-2}{2}}
	\end{align*} 
	where $C>0$ is a fixed constant , $a\in \mathbb{R}^N$ and $b\in (0,\infty)$ are parameters. Moreover,
	\[
	S=S_{H,L} ~(C(N,\mu))^{\frac{N-2}{2N -\mu}}.\]
\end{Lemma}
\begin{Lemma}\cite{yang}
	For $N\geq 3$ and $0<\mu<N$. Then 
	\begin{align*}
	\|.\|_{NL}:= \left(\ds \int_{\mathbb{R}^N}\int_{\mathbb{R}^N}\frac{|.|^{2^{*}_{\mu}}|.|^{2^{*}_{\mu}}}{|x-y|^{\mu}}~dxdy\right)^{\frac{1}{2.2^{*}_{\mu}}}
	\end{align*}
	defines a norm on $L^{2^*}(\mathbb{R}^N)$.
\end{Lemma}
The energy functional $\mc{J}_f: H^{1}_{0}(\Om)\rightarrow \mathbb{R}$ associated with the problem $(P_f)$ is  
\begin{equation}\label{nh25}
\mathcal{J}_f(u) = \frac12 \int_{\Om} |\nabla u|^2~dx - \frac{1}{2.2^*_{\mu}} \int_{\Om}\int_{\Om} \frac{|u^+(x)|^{2^{*}_{\mu}}|u^+(y)|^{2^{*}_{\mu}}}{|x-y|^{\mu}}~dxdy-\int_{\Om}fu ~dx,
\end{equation}
where $u^+= \text{max}(u,0)$. By using Hardy-Littlewood-Sobolev inequality \eqref{co9},  we have 
\begin{align*}
\left(\ds \int_{\Om}\int_{\Om}\frac{|u^+(x)|^{2^{*}_{\mu}}|u^+(y)|^{2^{*}_{\mu}}}{|x-y|^{\mu}}~dxdy\right)^{\frac{1}{2^{*}_{\mu}}}\leq C(N,\mu)^{\frac{2N-\mu}{N-2}}|u|_{2^*}^2. 
\end{align*}
 It is not difficult to show that  the functional $\mathcal{J}_f \in C^1(H_0^{1}(\Om),\mathbb{R})$ and moreover, if $\mu < \min \{4,N\}$ then  $\mathcal{J}_f \in C^2(H_0^{1}(\Om),\mathbb{R})$.
\begin{Definition}
	A function $u \in H_0^1(\Om)$ is called a weak solution of the problem $(P_f)$ if for all $v \in H_0^1(\Om)$ the following holds
	\begin{align*}
 \int_{\Om} \nabla u\cdot \na v ~dx -  \int_{\Om}\int_{\Om} \frac{|u^+(x)|^{2^{*}_{\mu}}|u^+(y)|^{2^{*}_{\mu}-1}v(y)}{|x-y|^{\mu}}~dxdy-\int_{\Om}fv ~dx=0
	\end{align*}
\end{Definition}
\begin{Definition}
	 For $c\in \mathbb{R}$, $\{u_n\}$ is a $(PS)_{c}$ sequence in $H_0^1(\Om)$ for $\mathcal{J}_f$ if $\mathcal{J}_f= c+o(1)$ and $\mathcal{J}^{\prime}_f(u_n)=o(1)$ strongly in $H^{-1}$ as $n\ra \infty$. We say $\mathcal{J}_f$ satisfies the $(PS)_{c}$ condition in $H_0^1(\Om)$ if every $(PS)_{c}$ sequence in  $H_0^1(\Om)$ has a convergent subsequence.
	\end{Definition}

\noi Since $\mathcal{J}_f$ is not bounded below on $H_0^1(\Om)$, it is worth to consider the Nehari manifold
\begin{align*} 
\mathcal{N}_f:=\{u \in H_0^1(\Om)\setminus\{0\}\; |\; u^+\not \equiv 0 \text{ and } \ld \mathcal{J}_{f}^{\prime}(u),u\rd =0\},
\end{align*}
where   $\ld\;, \;\rd $ denotes  the usual duality. We define $$
 \Upsilon_f= \ds \inf_{u \in \mathcal{N}_f}\mathcal{J}_f(u).$$ 
 Note that when $f(x)\equiv 0$, $\Upsilon_0(\Omega)$ is independednt of $\Om$ and $\Upsilon_0(\Omega):=\Upsilon_0 = \frac{N-\mu+2}{2(2N-\mu)}S_{H,L}^{\frac{2N-\mu}{N-\mu+2}}.$
  
\noi {\bf Notations:} Throughout the paper we will use the notation $\mathcal{J}_0=\mathcal{J},\; \mathcal{N}_0=\mathcal{N}, \|.\|=\|.\|_{H^{1}_{0}(\Om)}$ 

 $$a(u)= \ds \int_{\Om}|\na u |^2~dx \; \text{and} \;b(u)=\ds \int_{\Om}\int_{\Om}\frac{(u^+(x))^{2^{*}_{\mu}}(u^+(y))^{2^{*}_{\mu}}}{|x-y|^{\mu}}~dxdy$$.\\
An easy consequence of equation \eqref{co9} gives 
 $\mathcal{J}_f$ is coercive and bounded below on $\mathcal{N}_f$.
 \begin{Proposition}\label{Propnh1}
 	For any $u,v  \in H_0^1(\Om)$, we have 
 {\small
 	\begin{align*}
 	\int_{\Om}\int_{\Om} \frac{|u(x)|^{2^{*}_{\mu}}|v(y)|^{2^{*}_{\mu}}}{|x-y|^{\mu}}~dxdy \leq \left(\ds \int_{\Om}\int_{\Om}\frac{|u(x)|^{2^{*}_{\mu}}|u(y)|^{2^{*}_{\mu}}}{|x-y|^{\mu}}~dxdy\right)^{\frac{1}{2}} \left(\ds \int_{\Om}\int_{\Om}\frac{|v(x)|^{2^{*}_{\mu}}|v(y)|^{2^{*}_{\mu}}}{|x-y|^{\mu}}~dxdy\right)^{\frac{1}{2}}.
 	\end{align*}}
 \end{Proposition}
 \begin{proof}
 For details of the proof see \cite[Lemma 2.3]{systemchoq}.\QED
 	\end{proof}

 \begin{Lemma}\label{nhlem11}
   For each $u\in H_0^1(\Om)$ there exists a unique $t >0 $ such that $t u \in \mathcal{N}$. Moreover, there holds $ \Upsilon_0\leq  \ds \left(\frac{N-\mu+2}{2(2N-\mu)}\right) \left(\frac{ a(u)^{2^*_{\mu}}}{b(u)}\right)^{\frac{1}{2^*_{\mu}-1}}$.
 \end{Lemma}
 \begin{proof} Let $m_u(t)=\frac{t^2}{2}a(u)-\frac{t^{2.2^*_{\mu}}}{2.2^*_{\mu}}b(u) $ then on solving $m_u^{\prime}(t)=0$, we get  unique $t(u)= \left(\frac{a(u)}{b(u)}\right)^{\frac{1}{2(2^*_{\mu}-1)}}$ such that $t(u)u \in \mathcal{N}$. From the definition of $\Upsilon_0$, we have 
 $$
 \Upsilon_0\leq \mathcal{J}(t(u)u)= \left(\frac{1}{2}-\frac{1}{2.2^*_{\mu}}\right)\left(\frac{a(u)}{b(u)}\right)^{\frac{1}{2^*_{\mu}-1}}a(u)= \left(\frac{N-\mu+2}{2(2N-\mu)}\right)\left(\frac{ a(u)^{2^*_{\mu}}}{b(u)}\right)^{\frac{1}{2^*_{\mu}-1}}.$$
 \end{proof} \QED

\begin{Remark}\label{nhrem1}
	We remark that by Lemma 1.3 of \cite{yang}, $S_{H,L}$ is never achieved on bounded domain then if $u$ is a solution of the following equation 
	\begin{align*}
	-\De u & = \left(\int_{\Om}\frac{|u(y)|^{2^*_{\mu}}}{|x-y|^{\mu}}dy\right)|u|^{2^*_{\mu}-2}u  \; \text{ in }\;
	\Om,\qquad 	u  = 0 \;    \text{ on } \pa \Om ,
	\end{align*}
	then $\mathcal{J}(u)> \Upsilon_0=\frac{N-\mu+2}{2(2N-\mu)}S_{H,L}^{\frac{2N-\mu}{N-\mu+2}}.$
\end{Remark}

 \begin{Lemma} \label{nhlem14}
	A sequence $\{u_n\}$ is a $(PS)_{\Upsilon_0}$- sequence for $\mathcal{J}$ in $H_0^1(\Om)$ if and only if $\mathcal{J}(u_n)= \Upsilon_0+o_n(1)$ and $a(u_n)=b(u_n)+o_n(1)$.
\end{Lemma}
\begin{proof}
	Clearly, any $(PS)_{\Upsilon_0}$- sequence satisfies $a(u_n)=b(u_n)+o_n(1)$ and $\mathcal{J}(u_n)= \Upsilon_0+o_n(1)$. Conversely, let  $\mathcal{J}(u_n)= \Upsilon_0+o_n(1)$ and $a(u_n)=b(u_n)+o_n(1)$ then $\Upsilon_0=\mathcal{J}(u_n)= \frac{N-\mu+2}{2(2N-\mu)}b(u_n)+o_n(1)$ and hence we have 
	\begin{equation}\label{nh3}
	b(u_n)= D\Upsilon_0+o_n(1) \text{ where }  D=\frac{2(2N-\mu)}{N-\mu+2}.
\end{equation}		
 Define  $T_n(\psi)=\ds \int_{\Om}\int_{\Om}\frac{(u^+_n(x))^{2^{*}_{\mu}}(u^+_n(y))^{2^{*}_{\mu}-1}\psi(y)}{|x-y|^{\mu}}~dxdy$ for $\psi \in H_0^1(\Om)$ and  $n=1,2,\cdots$. 
	\textbf{Claim:} $\|T_n\|_{H^{-1}}=(D\Upsilon_0)^{\frac{1}{2}}+o_n(1)$.\\ For this let $\psi \in H_0^1(\Om)$ such that $\|\psi\|=1$ then by Lemma \ref{nhlem11}, we know that there exists a $t>0$ such that $a(t\psi)=b(t\psi)$. Therefore, $t= \|\psi\|_{NL}^{-\frac{2^*_{\mu}}{2^*_{\mu}-1}}$ and $\Upsilon_0\leq \frac{1}{D}\|\psi\|_{NL}^{-\frac{2.2^*_{\mu}}{2^*_{\mu}-1}}.$ This implies,
	\begin{equation}\label{nh2}
	\|\psi\|_{NL}\leq \left(\frac{1}{D\Upsilon_0}\right)^{\frac{2^*_{\mu}-1}{2.2^*_{\mu}}}.
\end{equation}	 
	Taking into account   equations \eqref{nh3}, \eqref{nh2}, Proposition \ref{Propnh1}    and employing H\"older's inequality, for each $n$, we have 
	\begin{equation*}
	\begin{aligned}
|T_n(\psi)|
& \leq \left( \ds \int_{\Om}\int_{\Om}\frac{(u^+_n(x))^{2^{*}_{\mu}}(u^+_n(y))^{2^{*}_{\mu}}}{|x-y|^{\mu}}~dxdy\right)^{\frac{2.2^*_{\mu}-1}{2.2^*_{\mu}}}\left( \ds \int_{\Om}\int_{\Om}\frac{|\psi(x)|^{2^{*}_{\mu}}|\psi(y)|^{2^{*}_{\mu}}}{|x-y|^{\mu}}~dxdy\right)^{\frac{1}{2.2^*_{\mu}}}\\
& = b(u_n)^{\frac{2.2^*_{\mu}-1}{2.2^*_{\mu}}}\|\psi\|_{NL}\\
& \leq  \left(\frac{1}{D\Upsilon_0}\right)^{\frac{2^*_{\mu}-1}{2.2^*_{\mu}}}(D\Upsilon_0+o_n(1))^{\frac{2.2^*_{\mu}-1}{2.2^*_{\mu}}}
 = (D\Upsilon_0)^{\frac{1}{2}}+o_n(1) \text{ as } n \ra \infty.
\end{aligned}
\end{equation*}
So, we get $\|T_n\|_{H^{-1}}\leq (D\Upsilon_0)^{\frac{1}{2}}+o_n(1)$. Moreover, $  T_n\left(\frac{u_n}{\|u_n\|}\right)= (b(u_n))^{\frac{1}{2}}=(D\Upsilon_0)^{\frac{1}{2}}+o_n(1)$. This implies $ \|T_n\|_{H^{-1}} =(D\Upsilon_0)^{\frac{1}{2}}+o_n(1)$. Hence the proof of  claim follows.

 Now, by Riesz representation theorem, for each $n$ there exists $v_n\in H_0^1(\Om)$ such that 
\begin{align*}
T_n(\psi)=\ld v_n,\psi\rd = \int_{\Om} \na v_n\cdot \na \psi ~dx \text{ and } \|v_n\|= \|T_n\|_{H^{-1}}=(D\Upsilon_0)^{\frac{1}{2}}+o_n(1). 
\end{align*}
Thus, $\ld v_n,u_n\rd= T_n(u_n)= b(u_n)=D\Upsilon_0+o_n(1) $. Hence,
\begin{align*}
\|u_n-v_n\|^2&= \|u_n\|^2-2\ld u_n,v_n\rd +\|v_n\|^2\\
& = D\Upsilon_0-2D\Upsilon_0+D\Upsilon_0+o_n
(1)
 = o_n(1) \text{ as } n\ra \infty.
\end{align*}	 
	For any  $\psi   \in H_0^1(\Om)$ with $\|\psi\|=1$, we have 
	\begin{align*}
	\ld \mathcal{J}^{\prime}(u_n),\psi \rd = \ds \int_{\Om}\na u_n \cdot \na \psi ~dx - T_n(\psi)= \ld u_n, \psi \rd - \ld v_n, \psi \rd = \ld u_n-v_n, \psi \rd.
	\end{align*}
Therefore, $\|\mathcal{J}^{\prime}(u_n)\|_{H^{-1}}\leq  \|u_n-v_n\|= o_n(1)$. It implies $\mathcal{J}^{\prime}(u_n) \ra 0 $ in $H^{-1}$.\QED
\end{proof}
\noi Clearly, $ \mathcal{N}_{f}$ contains every non zero solution of $(P_f)$ and we know that the Nehari manifold is closely
related to the behavior of the fibering maps $\phi_u: \mathbb R^+\rightarrow \mathbb R$
defined as $\phi_{u}(t)=\mathcal{J}_{f}(tu)$. 
It is easy to see that $tu\in \mathcal{N}_f$ if and only if $\phi_{u}^{\prime}(t)=0$ and elements of $\mathcal{N}_f$ correspond to stationary points of the fibering maps. It is natural to divide  $\mathcal{N}_f$ into the following sets
\begin{align*}
\mathcal{N}_f^+=:\{u \in \mathcal{N}_f |\phi_{u}^{\prime\prime}(1)>0 \},\; 
\mathcal{N}_f^-=:\{u \in \mathcal{N}_f |\phi_{u}^{\prime\prime}(1)<0 \},\text{and}\;
\mathcal{N}_f^0=:\{u \in \mathcal{N}_f |\phi_{u}^{\prime\prime}(1)=0 \}.
\end{align*}
We also denote the infimum over $\mc{N}_{f}^{+}$ and $\mc{N}_{f}^{-}$ as 
\begin{align}\label{nh19}
\Upsilon_f^+= \ds \inf_{u \in \mathcal{N}_f^+}\mathcal{J}_f(u) \qquad \Upsilon_f^-= \ds \inf_{u \in \mathcal{N}_f^-}\mathcal{J}_f(u).
\end{align}
\section{Existence of First Solution }
 In this section we prove the existence of first solution by showing the existence of minimizer for $\mc{J}_f$ over the Nehari manifold $\mc{N}_{f}$. First we state some Lemmas whose proof can be found in \cite{zampyang}.   We also prove some properties of the manifold $\mathcal{N}_f^+$.
\begin{Lemma}
If $f\in \hat{F}$ and  $	\|f\|_{H^{-1}} < e_{00}:= C_{N,\mu} S_{H,L}^{\frac{2^*_{\mu}}{2.2^*_{\mu}-2}}$ where 
	$ C_{N,\mu}= \left(\frac{1}{2.2^*_{\mu}-1}\right)^{\frac{2.2^*_{\mu}-1}{2.2^*_{\mu}-2}}(2.2^*_{\mu}-2)$ 	then 
	$\ds
\al_0:=\inf_{u \in E}\bigg\{C_{N,\mu}\|u\|^{\frac{2.2^*_{\mu}-1}{2^*_{\mu}-1}} -\int_{\Om} fu~dx \bigg\}$
	is acheived, where  
	\begin{align*}
	E:= \bigg\{ u \in H_0^1(\Om): \ds \int_{\Om}\int_{\Om} \frac{|u(x)|^{2^{*}_{\mu}}|u(y)|^{2^{*}_{\mu}}}{|x-y|^{\mu}}~dxdy=1 \bigg\}.
	\end{align*}
	\end{Lemma}
\begin{proof}
	Proof follows from Lemma 4.1 of \cite{zampyang}. Since we consider  $\la=0$ in equation 4.1 of \cite{zampyang}, our result holds for all $N\geq 3$.  \QED
	\end{proof}
\begin{Lemma}
	For every $u\in \mathcal{N}_f, u \not \equiv 0$ we have 
	$
	a(u)-  (2.2^*_{\mu}-1)b(u)\not = 0.$
	In particular, $\mathcal{N}_f^0= \{0\}$.
\end{Lemma}
\begin{Lemma}\label{nhlem3}
For each $u\in H_0^1(\Om)$ with $u^+\not \equiv 0$ the following holds:
\begin{enumerate}
\item [(a)]   There exists a unique $t^- =t^-(u) >0  $ such that $t^- u \in \mathcal{N}_f^-(\Om)$. In particular, 
\begin{align*}
t^- > \left(\frac{a(u)}{(2.2^*_{\mu}-1)b(u)}\right)^{\frac{1}{2.2^*_{\mu}-2}}:= t_{max}
\end{align*}
and $\mathcal{J}_f(t^-u)=\ds \max _{t\geq t_{max}}\mathcal{J}_f(tu)$.
\item [(b)]If $\ds \int_\Om fu>0$, then there exists unique $ t^+ \in(0,  t_{max})$ such that $t^+u \in \mathcal{N}_f^+(\Om)$ and $$\mathcal{J}_f(t^+u)=\ds \min _{0<t\leq t^-}\mathcal{J}_f(tu).$$
\item [(c)] $t^-(u)$ is a continuous function.
\item [(d)] $\mathcal{N}_f^-=\{u \in H_0^1(\Om)\setminus\{0\}\; |\; u^+\not \equiv 0 \text{ and } \frac{1}{\|u\|}t^-(\frac{u}{\|u\|})=1\}$.
\end{enumerate}
\end{Lemma}
\begin{Lemma}\label{nhlem28}
For each $u \in \mathcal{N}_f^+(\Om)$, we have $\int_{\Om}fu~dx >0 $ and $\mathcal{J}_f(u)<0.$ In particular, $\Upsilon_f(\Om)\leq \Upsilon_f^+(\Om)< 0$.
\end{Lemma}

\begin{Lemma}\label{nhlem4}
Let $u \in \mathcal{N}_f(\Om)$ be such that $\mathcal{J}_f(u)= \ds \min_{w\in  \mathcal{N}_f(\Om)}\mathcal{J}_f(w)= \Upsilon_f(\Om)$ then we have $\int_{\Om}fu~dx >0 $ and $u$ is a solution of $(P_f)$.
\end{Lemma}

\begin{Lemma} $\mathcal{J}_f$ has Palais-Smale sequences at each of the levels $\Upsilon_f(\Om)$, $\Upsilon_f^+(\Om)$ and $\Upsilon_f^-(\Om)$. 
\end{Lemma}

\begin{Lemma}\label{nhlem27}
	Let  $\{u_n\}\in \mathcal{N}_f$ be  a $(PS)_{\Upsilon_f(\Om)}$ sequence for  $\mathcal{J}_f$, then there exists a subsequence of $\{u_n\}$, still denoted by $\{u_n\},$ and a non-zero $u_1\in H_0^1(\Om)$ such that $u_n\ra u_1$ strongly in $H_0^1(\Om)$. Moreover,  $u_1 \in \mathcal{N}_f$ and solves of $(P_f)$. 
\end{Lemma}
\begin{proof}
	  $\mathcal{J}_f$ is bounded below and coercive  implies $\{u_n\}$ is bounded in $ H_0^1(\Om)$. So, there exists a subsequence still denoted by $\{u_n\}$ such that $u_n\rp u_1$ weakly in $ H_0^1(\Om)$. By \cite[Lemma 4.2]{choqcoron}, we have 
 $\mathcal{J}_f^{\prime}(u_1)=0$. In particular, $u_1 \in \mathcal{N}_f$ and  $\mathcal{J}_f(u_1)=\left(\frac{1}{2}-\frac{1}{2.2*_{\mu}}\right)a(u_1)-\left(1-\frac{1}{2.2*_{\mu}}\right)\ds \int_{\Om}fu_1~dx$. Now, using the fact that $a$ is weakly lower semi continuous we have 
	\begin{align*}
	\Upsilon_f(\Om)\leq \mathcal{J}_f(u_1)\leq \liminf_{n\ra \infty}\left(\frac{1}{2}-\frac{1}{2.2*_{\mu}}\right) a(u_n)-\lim_{n \ra \infty}\left(1-\frac{1}{2.2*_{\mu}}\right)\ds \int_{\Om}fu_n~dx= \Upsilon_f(\Om).
	\end{align*}
Consequently, we have $\Upsilon_f(\Om)=\mathcal{J}_f(u_1)$. Let $w_n=u_n-u_1$ then by  Lemma 4.1 of \cite{choqcoron}, Lemma 2.2 of \cite{yang} and the fact that  $\mathcal{J}_f^{\prime}(u_1)=0$,  we obtain
$\mathcal{J}_f(w_n)=\mathcal{J}_f(u_n)-\mathcal{J}_f(u_1)= o_n(1) \; \text{and}\;
\ld \mathcal{J}_f^{\prime}(w_n), \phi \rd = \ld \mathcal{J}_f^{\prime}(u_n), \phi \rd- \ld \mathcal{J}_f^{\prime}(u_1), \phi \rd +o_n(1)= o_n(1).$
Therefore, $\ld \mathcal{J}_f^{\prime}(w_n), w_n \rd= o_n(1)$. It implies $ \mathcal{J}_f(w_n)=\left(\frac{1}{2}-\frac{1}{2.2*_{\mu}}\right) a(w_n)- \int_{\Om}fw_n~dx = o_n(1)$ and since  $ \int_{\Om}fw_n~dx=o_n(1)$, we get $a(w_n)= o_n(1)$. Hence $u_n \ra u$ strongly in $H_0^1(\Om)$.\QED
	\end{proof}
\begin{Lemma}\label{nhlem34}
If $u$ be a solution of $(P_f)$ then  $u \in C^2(\overline{\Om})$. Moreover, $u$ is a positive solution.
\end{Lemma}
\begin{proof}
	Let  $u$ be a solution of $(P_f)$, and $G(x,u)=\left(\ds \int_{\Om}\frac{|u^+(y)|^{2^*_{\mu}}}{|x-y|^{\mu}}dy\right)|u^+|^{2^*_{\mu}-2}u+f$ and since  $f \in  \hat{F}$, we have $|G(x,u)| \leq  C + \left(\ds \int_{\Om}\frac{|u(y)|^{2^*_{\mu}}}{|x-y|^{\mu}}dy\right)|u|^{2^*_{\mu}-2}u$. Then by Lemma 4.4 of \cite{yangjmaa} we obtain $u\in L^\infty(\Om)$ and by the standard elliptic regularity $u \in C^2(\overline{\Om})$. Since $f\ge 0,$ we get $u\ge 0$ and by using  strong maximum principle,  $u$ is a positive solution of $(P_f)$.\QED
	\end{proof}

\begin{Lemma}\label{nhlem5}
	Let  $\mu<\min\{ 4,\; N \}$ and $k_0= \left(\frac{1}{2.2^*_{\mu}-1}\right)^{\frac{1}{2(2^*_{\mu}-1)}}S_{H,L}^{\frac{2^*_{\mu}}{2(2^*_{\mu}-1)}}$  and $f\in \hat{F}$, $	\|f\|_{H^{-1}} \leq e_{00}$  then 
	\begin{enumerate}
		\item $\mathcal{N}_f^+(\Om) \subset  B_{k_0}(0):=\{u \in H_0^1(\Om)\; |\; \|u\|<k_0\}$.
		\item $\mathcal{J}_f$ is strictly convex in $B_{k_0}(0)$.
	\end{enumerate}
\end{Lemma}
\begin{proof}
	\begin{enumerate}
		\item 	Let $u \in \mathcal{N}_f^+(\Om)$ then $\phi_u^{\prime}(1)=0$ and  $\phi_u^{\prime\prime}(1)>0$. That is, $a(u)= b(u)+\int_{\Om}fu~dx$ and $a(u)>(2.2^*_{\mu}-1)b(u)$. Therefore, $a(u) = b(u)+\int_{\Om}fu~dx  < \frac{1}{(2.2^*_{\mu}-1)}a(u)+\int_{\Om}fu~dx$. It implies $
		\left(1-\frac{1}{(2.2^*_{\mu}-1)}\right)a(u)\leq \|f\|_{H^{-1}}\|u\|.$
		So, 
		\begin{align*}
		\|u\| & \leq 	\frac{(2.2^*_{\mu}-1)}{2(2^*_{\mu}-1)}\|f\|_{H^{-1}}\\
		&\leq  \frac{(2.2^*_{\mu}-1)}{2(2^*_{\mu}-1)}C_{N,\mu} S_{H,L}^{\frac{2^*_{\mu}}{2.2^*_{\mu}-2}}
		= \left(\frac{1}{2.2^*_{\mu}-1}\right)^{\frac{1}{2(2^*_{\mu}-1)}}S_{H,L}^{\frac{2^*_{\mu}}{2(2^*_{\mu}-1)}}= k_0.
		\end{align*}
		\item By using H\"olders inequality and equation 
		\eqref{nh5}, we have 
		\begin{equation}\label{nh6}
		\begin{aligned}
		\ds \int_{\Om}\int_{\Om}\frac{(u^+(x))^{2^{*}_{\mu}-1}(u^+(y))^{2^{*}_{\mu}-1}z(x)z(y)}{|x-y|^{\mu}}~dxdy &  \leq b(u)^{\frac{2^*_{\mu}-1}{2^*_{\mu}}}\|z\|_{NL}^2\\
		& \leq S_{H,L}^{-(2^*_{\mu}-1)}a(u)^{(2^*_{\mu}-1)}S_{H,L}^{-1}a(z)\\
		& = S_{H,L}^{-2^*_{\mu}}a(u)^{(2^*_{\mu}-1)}a(z).
		\end{aligned}
		\end{equation}
\noi	Again  using H\"olders inequality, Proposition \ref{Propnh1} and  
\eqref{nh5}, we have 
\begin{equation}\label{nh7}
\begin{aligned}
\ds \int_{\Om}\int_{\Om}\frac{(u^+(x))^{2^{*}_{\mu}}(u^+(y))^{2^{*}_{\mu}-2}z^2(y)}{|x-y|^{\mu}}~dxdy&  \leq b(u)^{\frac{2^*_{\mu}-1}{2^*_{\mu}}}\|z\|_{NL}^2 \leq  S_{H,L}^{-2^*_{\mu}}a(u)^{(2^*_{\mu}-1)}a(z).
\end{aligned}
\end{equation}
	From equations 	\eqref{nh6}, \eqref{nh7} and definition of $\mathcal{J}_f^{\prime\prime}(u)(z,z)$, we get
 \begin{align*}
	\mathcal{J}_f^{\prime\prime}(u)(z,z)= & a(z)- 2^*_{\mu}\ds \int_{\Om}\int_{\Om}\frac{(u^+(x))^{2^{*}_{\mu}-1}(u^+(y))^{2^{*}_{\mu}-1}z(x)z(y)}{|x-y|^{\mu}}~dxdy\\
	& \quad - (2^*_{\mu}-1)\ds \int_{\Om}\int_{\Om}\frac{(u^+(x))^{2^{*}_{\mu}}(u^+(y))^{2^{*}_{\mu}-2}z^2(y)}{|x-y|^{\mu}}~dxdy\\
	& \geq a(z)\left(1-  2^*_{\mu}S_{H,L}^{-2^*_{\mu}}a(u)^{(2^*_{\mu}-1)}
	- (2^*_{\mu}-1) S_{H,L}^{-2^*_{\mu}}a(u)^{(2^*_{\mu}-1)}\right)\\
	& =  a(z)\left(1-(2.2^*_{\mu}-1) S_{H,L}^{-2^*_{\mu}}a(u)^{(2^*_{\mu}-1)}\right) \\
	& > a(z)\left(1-\frac{(2.2^*_{\mu}-1)}{(2.2^*_{\mu}-1)} \right)=0
		\end{align*}
		for $u \in B_{k_0}(0)\setminus\{0\} $. Then $\mathcal{J}_f^{\prime\prime}(u)$ is positive definite for $u \in B_{k_0}(0) $ and $\mathcal{J}_f(u)$ is strictly positive on $B_{k_0}(0)$.\QED
	\end{enumerate}
	\end{proof}
\begin{Lemma}\label{nhlem10}
	It holds that 
 $u_1 \in \mathcal{N}_f^+$ and $\mathcal{J}_f(u_1)= \Upsilon_f^+(\Om)=\Upsilon_f(\Om)$.
		 Moreover,  $u_1$ is the unique critical point of $\mathcal{J}_f$ in $B_{k_0}(0)$ and 
		 $u_1$ is a local minimum of $\mathcal{J}_f$ in $H_0^1(\Om)$.
	
\end{Lemma}
\begin{proof} Using the proof of Theorem 1.3 of \cite{zampyang}, we have  $\ds \int_{\Om}fu_1~ dx >0$.
		Now  if  $u_1 \in \mathcal{N}_f^-$ then there exists a unique $t^-(u_1)=1> t_{max}>t^+(u_1)>0$ such that $t^+(u_1)u_1 \in \mathcal{N}_f^+$ then by Lemma \ref{nhlem3} (2) we have 
		\begin{align*}
		\Upsilon_f(\Om) \leq \Upsilon_f^+(\Om) \leq \mathcal{J}_f(t^+(u_1)u_1)< \mathcal{J}_f(t^-(u_1)u_1)=  \mathcal{J}_f(u_1)= \Upsilon_f(\Om) .
		\end{align*}
	which is a contradiction. It implies $u_1\in \mathcal{N}_f^+$ and $\Upsilon_f^+(\Om)\leq \mathcal{J}_f(u_1)= \Upsilon_f(\Om)\leq \Upsilon_f^+(\Om)$ that is, $\mathcal{J}_f(u_1)= \Upsilon_f(\Om)= \Upsilon_f^+(\Om)$.
	  Using  Lemma \ref{nhlem4} and Lemma \ref{nhlem5}, we get  $u_1$ is the unique critical point of $\mathcal{J}_f$ in $B_{k_0}(0)$  and the proof of local minimum follows from   \cite[Lemma 3.2]{zampyang}. \QED
	
	\end{proof}

\begin{Lemma}\label{nhlem6}
	Let $\mu<\min\{ 4,\; N \}$ and  $u \in H_0^1(\Om)$ be a critical point of $\mathcal{J}_f$ then either $u \in \mathcal{N}_f^-$ or $u =u_1 $.
\end{Lemma}
\begin{proof}
	If $u \in H_0^1(\Om)$ be a critical point of $\mathcal{J}_f$ then $u \in \mathcal{N}_f= \mathcal{N}_f^+ \cup \mathcal{N}_f^-$. Now using the fact that $\mathcal{N}_f^+ \cap \mathcal{N}_f^- =\emptyset $ and $\mathcal{N}_f^+ \subset B_{k_0}(0)$	we have either $u \in \mathcal{N}_f^-$ or $u =u_1 $. \QED
	\end{proof}

\section{Asymptotic estimates and Palais-Smale Analysis} 
In this section we shall prove  that the functional $\mathcal{J}_f$ satisfies Palais-Smale condition strictly below the first critical level and (strictly) between the first and second critical levels. To start with  we shall prove several new estimates on the nonlinearity.

 It is known from Lemma \ref{nhlem7} that the best constant $S_{H,L}$
is achieved by the function
\begin{align*}
u(x)=S^{\frac{(N-\mu)(2-N)}{4(N-\mu+2)}}(C(N,\mu))^{\frac{2-N}{2(N-\mu+2)}} \frac{(N(N-2))^{\frac{N-2}{4}}}{\left(1+|x|^2\right)^{\frac{N-2}{2}}}, 
\end{align*} 
which is a solution of the problem $- \Delta u= (|x|^{-\mu}*|u|^{2^*_{\mu}})|u|^{2^*_{\mu}-1} \text{ in } \mathbb{R}^N$ with 
\begin{align*}
\ds \int_{\mathbb{R}^N}|\nabla u|^2~dx=\ds \int_{\mathbb{R}^N}\int_{\mathbb{R}^N}\frac{|u(x)|^{2^{*}_{\mu}}|u(y)|^{2^{*}_{\mu}}}{|x-y|^{\mu}}~dxdy=S_{H,L}^{\frac{2N-\mu}{N-\mu+2}}.
\end{align*}
We may assume $R_1= \rho,\; R_2= 1/\rho$ for $\rho \in (0,\frac{1}{2})$. Now, define  $\upsilon_\rho \in C_c^{\infty}(\mathbb{R}^N)$ such that $0 \leq \upsilon_\rho(x) \leq 1$ for all  $x \in \mathbb{R}^N$, radially symmetric and $$
\upsilon_\rho(x)= \left\{
\begin{array}{lr}
0 \;\quad 0 <|x|<\ds \frac{3\rho}{2}, \\
1\; \quad 2\rho\leq |x|\leq \ds \frac{1}{2\rho},\\
0 \; \quad |x|\geq \ds \frac{3}{4\rho}, 
\end{array}
\right.
$$
and 
\begin{align*}
u^{\e}_\sigma(x)= S^{\frac{(N-\mu)(2-N)}{4(N-\mu+2)}}C(N,\mu)^{\frac{2-N}{2(N-\mu+2)}} \frac{(N(N-2)\e^2)^{\frac{N-2}{4}}}{\left(\e^2+|x-(1-\e)\sigma|^2\right)^{\frac{N-2}{2}}},
\end{align*}
where $\sigma\in \mathbb{S}^{N-1}:= \{x\in\mathbb{R}^N: |x|=1 \}$, $0<\e\leq  1$.
Set
\begin{align}\label{nh8}
g_\rho^{\e,\sigma}(x):= \upsilon_\rho(x)u^{\e}_\sigma(x) \in H_0^1(\Om).
\end{align}

\begin{Lemma}\label{nhlem8}
	\begin{enumerate}
		\item[(i)] $a(g_\rho^{\e,\sigma})=b(g_\rho^{\e,\sigma})=S_{H,L}^{\frac{2N-\mu}{N-\mu+2}}+o_\e(1) $ uniformly in $\sigma$ as $\e \ra 0$.
		\item[(ii)] $\mathcal{J}(g_\rho^{\e,\sigma})= \frac{N-\mu+2}{2(2N-\mu)}S_{H,L}^{\frac{2N-\mu}{N-\mu+2}}+o_\e(1) $ uniformly in $\sigma$ as $\e \ra 0$.
		\item[(iii)] $g_\rho^{\e,\sigma} \rp 0 $ weakly in $H_0^1(\Om)$ uniformly in $\sigma$ as $\e \ra 0$.
	\end{enumerate}
\end{Lemma}
\begin{proof}
	\begin{enumerate}
		\item [(i)] 
 Observe the fact that there exist constants $d_1,\; d_2>0$ such that
 \begin{equation}\label{nh24}
 d_1 < |x-(1-\e)\sigma|<d_2 \text{ for all } x\in B_{2 \rho }
\text{ whenever } \e<1-2\rho.
 \end{equation}
  \begin{align*}
  \| \na g_\rho^{\e,\sigma}\|_{L^2(\mathbb{R}^N)}- \| \na u_{\e}^{\sigma}\|_{L^2(\mathbb{R}^N)}  \leq &
 \int_{(\mathbb{R}^N\setminus B_{\frac{1}{2 \rho}})\cup B_{2\rho}}| \na u_{\e}^{\sigma}|^2~dx + \rho^{-2}\int_{ B_{2\rho}}|  u_{\e}^{\sigma}|^2~dx \\ 
 &+ \rho^{2}\int_{B_{\frac{3}{4 \rho}}\setminus B_{\frac{1}{2 \rho}}}|  u_{\e}^{\sigma}|^2~dx\\
\leq C & \e^{N-2}  \int_{(\mathbb{R}^N\setminus B_{\frac{1}{2 \rho}})\cup B_{2\rho}}\frac{|x-(1-\e)\sigma|^2}{|x-(1-\e)\sigma|^{2N}} ~dx\\
&+
 C \e^{N-2} \int_{ B_{2\rho} \cup B_{\frac{3}{4 \rho}}\setminus B_{\frac{1}{2 \rho}}}\frac{dx}{|x-(1-\e)\sigma|^{2(N-2)}} = O(\e^{N-2}).
\end{align*}
Thus, $\| \na g_\rho^{\e,\sigma}\|_{L^2(\mathbb{R}^N)}= \| \na u_{\e}^{\sigma}\|_{L^2(\mathbb{R}^N)} +o_\e(1)=  S_{H,L}^{\frac{2N-\mu}{N-\mu+2}}+o_\e(1).$\\ 
Next we will prove that $b(g_\rho^{\e,\sigma})=S_{H,L}^{\frac{2N-\mu}{N-\mu+2}}+o_\e(1) $ uniformly in $\sigma$ as $\e \ra 0$. For this  consider
\begin{align} \label{ksnew}
 \int_{\mathbb{R}^N}\int_{\mathbb{R}^N} &\frac{ |g_\rho^{\e,\sigma}(x)|^{2^{*}_{\mu}}|g_\rho^{\e,\sigma}(y)|^{2^{*}_{\mu}}}{|x-y|^{\mu}}~dxdy-\int_{\mathbb{R}^N}\int_{\mathbb{R}^N}\frac{ |u_{\e}^{\sigma}(x)|^{2^{*}_{\mu}}|u_{\e}^{\sigma}(y)|^{2^{*}_{\mu}}}{|x-y|^{\mu}}~dxdy\nonumber\\
 =&\int_{\mathbb{R}^N}\int_{\mathbb{R}^N}\frac{ ( |\upsilon_\rho(x)|^{2^{*}_{\mu}}|\upsilon_\rho(y)|^{2^{*}_{\mu}}-1)|u_{\e}^{\sigma}(x)|^{2^{*}_{\mu}}|u_{\e}^{\sigma}(y)|^{2^{*}_{\mu}}}{|x-y|^{\mu}}~dxdy \nonumber\\
\leq& C\left(\int_{ B_{2 \rho}}\int_{B_{2 \rho}}+\int_{B_{\frac{1}{2 \rho}}\setminus  B_{2 \rho}}\int_{ B_{2 \rho}}+\int_{ B_{\frac{1}{2 \rho}}\setminus  B_{2 \rho}}\int_{\mathbb{R}^N \setminus B_{\frac{1}{2 \rho}} }\right.\nonumber\\ & \left. \quad\quad+\int_{ \mathbb{R}^N \setminus B_{\frac{1}{2 \rho}}}\int_{B_{2 \rho}}+\int_{\mathbb{R}^N \setminus B_{\frac{1}{2 \rho}}}\int_{\mathbb{R}^N \setminus B_{\frac{1}{2 \rho}}} \right) \frac{|u_{\e}^{\sigma}(x)|^{2^*_{\mu}}|u_{\e}^{\sigma}(y)|^{2^*_{\mu}} } {|x-y|^{\mu}} ~dxdy,\nonumber\\
   =& C \ds  \sum_{i=1}^{i=5}J_i,
\end{align}

Taking into account the definition of $u_{\e}^{\sigma}$,  \eqref{nh24} and  Hardy-Littlewood-Sobolev inequality, we have the following estimates: Let $\xi_\e(x)=\frac{\e^N}{(\e^2+|x-(1-\e)\sigma|^2)^N}$,  then
\begin{align*}
J_1
& \leq C(N,\mu) \left(\int_{ B_{2 \rho}}  S^{\frac{-N(N-\mu)}{2(N-\mu+2)}}C(N,\mu)^{\frac{-N}{(N-\mu+2)}} (N(N-2))^{\frac{N}{2}} \xi_\e(x) ~dx \right)^{\frac{2N-\mu}{N}}\\
& \leq C \e^{2N-\mu}\left(\int_{B_{2 \rho}}\frac{dx}{|x-(1-\e)\sigma|^{2N}}\right)^{\frac{2N-\mu}{N}} \leq C \e^{2N-\mu} \left(\int_{B_{2 \rho}}dx\right)^{\frac{2N-\mu}{N}}= O(\e^{2N-\mu}),\\
J_2
& \leq C \left(\int_{B_{\frac{1}{2 \rho}}\setminus  B_{2 \rho}}\xi_\e(x)~dx \right)^{\frac{2N-\mu}{2N}}\left(
\int_{B_{2 \rho}}\xi_\e(x) ~dx \right)^{\frac{2N-\mu}{2N}}\\
& \leq C \e^{\frac{2N-\mu}{2}}\left(\int_{B_{2 \rho}}\frac{dx}{|x-(1-\e)\sigma|^{2N}}\right)^{\frac{2N-\mu}{2N}} = O(\e^{\frac{2N-\mu}{2}}),\\
J_3
& \leq C \left(\int_{B_{\frac{1}{2 \rho}}\setminus  B_{2 \rho}}\xi_\e(x) \right)^{\frac{2N-\mu}{2N}}\left(
\int_{\mathbb{R}^N \setminus B_{\frac{1}{2 \rho}}}\xi_\e(x)~dx \right)^{\frac{2N-\mu}{2N}}\\
& 
\leq C \e^{\frac{2N-\mu}{2}} \left(\int_{\mathbb{R}^N \setminus B_{\frac{1}{2 \rho}}}\frac{dx}{|x-(1-\e)\sigma|^{2N}}\right)^{\frac{2N-\mu}{2N}} = O(\e^{\frac{2N-\mu}{2}}),\\
\end{align*}
\begin{align*}
J_4& \leq C \left(\int_{\mathbb{R}^N \setminus B_{\frac{1}{2 \rho}}} \xi_\e(x)~dx \right)^{\frac{2N-\mu}{2N}}\left(
\int_{B_{2 \rho}} \xi_\e(x) ~dx\right)^{\frac{2N-\mu}{2N}}\\
& \leq C \e ^{2N-\mu} \left(\int_{\mathbb{R}^N \setminus B_{\frac{1}{2 \rho}}}\frac{dx}{|x-(1-\e)\sigma|^{2N}}
\int_{B_{2 \rho}}\frac{dx}{|x-(1-\e)\sigma|^{2N}}\right)^{\frac{2N-\mu}{2N}}= O(\e ^{2N-\mu}),\\
J_5
& \leq C \left(\int_{\mathbb{R}^N \setminus B_{\frac{1}{2 \rho}}} \xi_\e(x) ~dx \right)^{\frac{2N-\mu}{N}}
 \leq C \e ^{2N-\mu} \left(\int_{\mathbb{R}^N \setminus B_{\frac{1}{2 \rho}}}\frac{dx}{|x-(1-\e)\sigma|^{2N}}\right)^{\frac{2N-\mu}{N}}\\
 &= O(\e ^{2N-\mu}).
\end{align*}
\noi Therefore, $b(g_\rho^{\e,\sigma})- \ds \int_{\mathbb{R}^N}\int_{\mathbb{R}^N}\frac{ |u_{\e}^{\sigma}(x)|^{2^{*}_{\mu}}|u_{\e}^{\sigma}(y)|^{2^{*}_{\mu}}}{|x-y|^{\mu}}~dxdy \ra 0 $ as $\e \ra 0$ that is, $b(g_\rho^{\e,\sigma})\ra S_{H,L}^{\frac{2N-\mu}{N-\mu+2}}$ as $\e \ra 0$ and completes the proof of 1.
\noi \item[(ii)] Result follows from the definition of $\mathcal{J}$ and by $(i)$.
\item[(iii)] Assume by contradiction, $g_\rho^{\e,\sigma} \rp g_1 \not \equiv 0 $ weakly in $H_0^1(\Om)$ then  $g_\rho^{\e,\sigma} \ra g_1 $ strongly  in $L^2(\Om)$. Then by using the inequality  $r^{2(N-2)}+s^{2(N-2)}\leq (r^2+s^2)^{N-2}$ for all $r,\; s \geq 0$, we have 
\begin{align*}
0\leq \int_{\Om}|g_\rho^{\e,\sigma}|^2~dx & \leq C\int_{\frac{3\rho}{2}\leq |x|\leq \frac{3}{4\rho}}\frac{\e^{N-2}}{(\e^2+|x-(1-\e)\sigma|^2)^{N-2}}~dx\\
& = C\int_{\frac{3\rho}{2}\leq |y+(1-\e)\sigma|\leq \frac{3}{4\rho}}\frac{\e^{N-2}}{\e^{2(N-2)}+|y|^{2(N-2)}}~dy\\
& \leq C\int_ 0^{ \frac{3}{4\rho}+(1-\e)}\frac{\e^{N-2}r^{N-1}}{\e^{2(N-2)}+r^{2(N-2)}}~dy \ra 0
\end{align*}
It yields a contradiction. Hence results follows.\QED
\end{enumerate}
	\end{proof}

\begin{Lemma}\label{nhlem19}
	Let $\sigma\in  \mathbb{S}^{N-1}$ and $\e \in (0,1)$, then the following holds:
	\begin{enumerate}
		\item[(i)] $\ds \lim_{\rho\ra 0}\ds \sup _{\sigma \in \mathbb{S}^{N-1}, \e\in (0,1]}\|\na (g_\rho^{\e,\sigma}-u_{\e}^{\sigma})\|^2_{L^2(\mathbb{R}^N)}=0$.
		\item [(ii)]$\ds \lim_{\rho\ra 0}\ds \sup _{\sigma \in \mathbb{S}^{N-1}, \e\in (0,1]}\|g_\rho^{\e,\sigma}\|_{NL}^{2.2^*_{\mu}}=\|u_{\e}^{\sigma}\|_{NL}^{2.2^*_{\mu}}$.
	\end{enumerate}
\end{Lemma}
\begin{proof}
	\begin{enumerate}
		\item [(i)]
	 Consider 
	\begin{equation}\label{nh12}
	\begin{split}
	\int_{\mathbb{R}^N}|\na g_\rho^{\e,\sigma}-\na u_{\e}^{\sigma} |^2 dx & \leq 2 \int_{\mathbb{R}^N}|u_{\e}^{\sigma}(x) \na \upsilon _\rho(x)|^2 ~ dx +2 \int_{\mathbb{R}^N}|\na u_{\e}^{\sigma}(x)  \upsilon _\rho(x)- \na u_{\e}^{\sigma}(x)|^2 ~ dx\\ & \leq  C \left( \rho^{-2}\int_{B_{2\rho}}|u_{\e}^{\sigma}(x)|^2 ~ dx+ \int_{B_{2\rho}}|\na u_{\e}^{\sigma}(x)|^2 ~ dx \right) \\ & \qquad +
	C\left(\rho^2 \int_{B_{\frac{3}{4\rho}} \setminus B_{\frac{1}{2\rho}}}|u_{\e}^{\sigma}(x)|^2 ~ dx+\int_{\mathbb{R}^N \setminus B_{\frac{1}{2\rho}}} |\na u_{\e}^{\sigma}(x)|^2 ~ dx \right). 
	\end{split}
	\end{equation}
	From the definition of $u_{\e}^{\sigma}$, we have the following estimates  
	 \begin{align*}
	\rho^{-2}\int_{B_{2\rho}}|u_{\e}^{\sigma}(x)|^2 ~ dx & \leq  C \rho^{-2} \int_{B_{2\rho}} ~ dx \leq C\rho^{N-2},\\
	\int_{B_{2\rho}}|\na u_{\e}^{\sigma}(x)|^2 ~ dx &\leq C \int_{B_{2\rho}} |x-t\sigma|~ dx \leq C \int_{B_{2\rho}} ~ dx \leq C\rho^N,\\
	\rho^2 \int_{B_{\frac{3}{4\rho}} \setminus B_{\frac{1}{2\rho}}}|u_{\e}^{\sigma}(x)|^2 ~ dx &\leq C\rho^2 \int_{B_{\frac{3}{4\rho}} \setminus B_{\frac{1}{2\rho}}}\frac{1}{|x|^{2N-4}} ~ dx \leq C\rho^{N-2},\\
	\int_{\mathbb{R}^N \setminus B_{\frac{1}{2\rho}}} |\na u_{\e}^{\sigma}(x)|^2 ~ dx   &\leq C \int_{\mathbb{R}^N  \setminus B_{\frac{1}{2\rho}}} \frac{1}{|x|^{2N-2}} ~ dx \leq C \rho^{N-2}. 
	\end{align*}
	
	\noi Therefore, from  above estimates and \eqref{nh12}, we obtain desired result. 
	\item [(ii)] Consider 
	\begin{align*}
	\|g_\rho^{\e,\sigma}\|_{NL}^{2.2^*_{\mu}}-\|u_{\e}^{\sigma}\|_{NL}^{2.2^*_{\mu}}
	& = \int_{\mathbb{R}^N}\int_{\mathbb{R}^N}\frac{(\upsilon_\rho^{2^*_{\mu}}(x)\upsilon_\rho^{2^*_{\mu}}(y)-1)|u_{\e}^{\sigma}(x)|^{2^*_{\mu}}|u_{\e}^{\sigma}(y)|^{2^*_{\mu}} } {|x-y|^{\mu}} ~dxdy\\
	& \leq C \sum_{i=1}^{5} J_i, 
	\end{align*}
	where $J_i$ are defined in equation \eqref{ksnew}. Using the Hardy-Littlewood-Sobolev inequality, we have the following estimates: Recall $\xi_\e(x)=\frac{\e^N  }{(\e^2+|x-(1-\e)\sigma|^2)^N}$, then 
	\begin{align*}
	J_1
	& \leq C(N,\mu) \left(\int_{B_{2\rho}} \xi_\e(x) ~dx \right)^{\frac{2N-\mu}{N}} \leq C \left(\int_{{B_{2\rho}}}  dx \right)^{\frac{2N-\mu}{N}} \leq C \rho^{2N-\mu},  \\
	J_2
	& \leq C(N,\mu) \left(\int_{B_{\frac{1}{2 \rho}}\setminus  B_{2 \rho}}\xi_\e(x) ~dx \right)^{\frac{2N-\mu}{2N}}\left(
	\int_{ B_{2 \rho}} \xi_\e(x) ~dx\right)^{\frac{2N-\mu}{2N}}\\
	& \leq C \left(\int_{{B_{2\rho}}}dx \right)^{\frac{2N-\mu}{N}} \leq C \rho^{\frac{2N-\mu}{2}},
	\end{align*}
	\begin{align*}
	J_3& 
	\leq C(N,\mu) \left(\int_{B_{\frac{1}{2 \rho}}\setminus  B_{2 \rho}} \xi_\e(x) dx\right)^{\frac{2N-\mu}{2N}}\left(
	\int_{\mathbb{R}^N \setminus B_{\frac{1}{2 \rho}}} \xi_\e(x) dx\right)^{\frac{2N-\mu}{2N}}\\
	&  \leq C    \left(\int_{\mathbb{R}^N \setminus B_{\frac{1}{2 \rho}}} \frac{ dx}{|x-(1-\e)\sigma|^{2N}}\right)^{\frac{2N-\mu}{2N}} \\
	&   = \left(\int_{|y+(1-\e)\sigma|\geq \frac{1}{2 \rho}} \frac{ dy}{|y|^{2N}}\right)^{\frac{2N-\mu}{2N}} \leq  \left(\int_{|y|\geq \frac{1}{2 \rho}-1} \frac{ dy}{|y|^{2N}}\right)^{\frac{2N-\mu}{2N}} 	\leq C \left(\frac{(2\rho)^N}{1-(2\rho)^N}\right)^{\frac{2N-\mu}{2N}},
	\end{align*}
	Now using the same estimates as above we can easily obtain
	\begin{align*}
	J_4&  \leq C \rho^{\frac{2N-\mu}{2}} \text{ and } J_5 
	\leq C \left(\frac{(2\rho)^N}{1-(2\rho)^N}\right)^{\frac{2N-\mu}{N}}.
	\end{align*}
	Hence $ \ds \sup _{\sigma \in \mathbb{S}^{N-1}, \e \in (0,1]}\left(\|g_\rho^{\e,\sigma}\|_{NL}^{2.2^*_{\mu}}-\|u_{\e}^{\sigma}\|_{NL}^{2.2^*_{\mu}}\right)\ra 0$  as  $\rho \ra 0$ and completes the proof.\QED
\end{enumerate}
	\end{proof}

\begin{Lemma}\label{nhlem9} The following asymptic estimates hold:
	\begin{itemize}
		\item[(i)] $a(g_\rho^{\e,\sigma})\leq  S_{H,L}^{\frac{2N-\mu}{N-\mu+2}}+O(\e^{N-2})$.
		\item [(ii)] $b(g_\rho^{\e,\sigma})\leq S_{H,L}^{\frac{2N-\mu}{N-\mu+2}}+O(\e^N)$.
		\item [(iii)] $b(g_\rho^{\e,\sigma})\geq S_{H,L}^{\frac{2N-\mu}{N-\mu+2}}-O(\e^{\frac{2N-\mu}{2}})$.
	\end{itemize}
\end{Lemma}
\begin{proof}
	Part $(i)$ follows from  Lemma \ref{nhlem8} $(i).$ For part $(ii)$ we will first estimate the integral $\ds \int_{\Om}|g_\rho^{\e,\sigma}|^{2^*}~dx $. Since 
	\begin{align*}
\ds \int_{\Om}|g_\rho^{\e,\sigma}|^{2^*}~dx \leq C \ds \int_{B_{\frac{3}{4\rho}} \setminus B_{\frac{3\rho}{2}}}|u^\sigma_{\e}|^{2^*}~dx\leq    \ds \int_{B_{\frac{3}{4\rho}} \setminus B_{\frac{1}{2\rho}}}|u^\sigma_{\e}|^{2^*}~dx+   \ds \int_{B_{\frac{1}{2\rho}} \setminus B_{\frac{3\rho}{2}}}|u^\sigma_{\e}|^{2^*}~dx
	\end{align*}
and \begin{align*}
& \ds \int_{B_{\frac{3}{4\rho}} \setminus B_{\frac{1}{2\rho}}}|u^\sigma_{\e}|^{2^*}~dx \leq C \e ^{N}\int_{B_{\frac{3}{4\rho}} \setminus B_{\frac{1}{2\rho}}}\frac{dx}{|x-(1-\e)\sigma|^{2N}}= O(\e^N),\\
&  \ds \int_{B_{\frac{1}{2\rho}} \setminus B_{\frac{3\rho}{2}}}|u^\sigma_{\e}|^{2^*}~dx \leq  \ds \int_{\mathbb{R}^N}|u^\sigma_{\e}|^{2^*}~dx= S^{\frac{N}{N-\mu+2}}C(N,\mu)^{\frac{-N}{N-\mu+2}}.
\end{align*}
It implies $\ds \int_{\Om}|g_\rho^{\e,\sigma}|^{2^*}~dx\leq S^{\frac{N}{N-\mu+2}}C(N,\mu)^{\frac{-N}{N-\mu+2}}+ O(\e^N)$ and now using this and Hardy-Littlewood-Sobolev inequality we have
 \begin{align*}
b(g_\rho^{\e,\sigma})& = \int_{ \Om}\int_{\Om}\frac{|g_\rho^{\e,\sigma}(x)|^{2^*_{\mu}}|g_\rho^{\e,\sigma}(y)|^{2^*_{\mu}} } {|x-y|^{\mu}} ~dxdy \\
& \leq C(N,\mu) \left(\ds \int_{\Om}|g_\rho^{\e,\sigma}|^{2^*}~dx\right)^{\frac{2N-\mu}{N}} \\
& \leq  C(N,\mu) \left(S^{\frac{N}{N-\mu+2}}C(N,\mu)^{\frac{-N}{N-\mu+2}}+ O(\e^N)\right)^{\frac{2N-\mu}{N}}\leq S_{H,L}^{\frac{2N-\mu}{N-\mu+2}}+O(\e^N).
\end{align*}
This proves part $(ii).$
Now to prove part $(iii),$ consider
\begin{align*}
b(g_\rho^{\e,\sigma})& = \int_{ \Om}\int_{\Om}\frac{|g_\rho^{\e,\sigma}(x)|^{2^*_{\mu}}|g_\rho^{\e,\sigma}(y)|^{2^*_{\mu}} } {|x-y|^{\mu}} ~dxdy \\
& \geq \int_{ B_{\frac{1}{2\rho}}\setminus B_{2\rho}}\int_{B_{\frac{1}{2\rho}}\setminus B_{2\rho}}\frac{|g_\rho^{\e,\sigma}(x)|^{2^*_{\mu}}|g_\rho^{\e,\sigma}(y)|^{2^*_{\mu}} } {|x-y|^{\mu}} ~dxdy \\
& = \int_{\mathbb{R}^N}\int_{\mathbb{R}^N}\frac{ |u_{\e}^{\sigma}(x)|^{2^{*}_{\mu}}|u_{\e}^{\sigma}(y)|^{2^{*}_{\mu}}}{|x-y|^{\mu}}~dxdy - \ds \sum_{i=1}^{i=5}J_i,
\end{align*}
	where $J_i$ are defined in equation \eqref{ksnew}. Using the proof of  Lemma \ref{nhlem8}$(i)$ and the fact that $\|u_{\e}^{\sigma}\|_{NL}^{2.2^*_{\mu}}= S_{H,L}^{\frac{2N-\mu}{N-\mu+2}} +o_\e(1)$, we obtain the required result.
\QED
	\end{proof}

\begin{Lemma}\label{nhlem33}
If  $\mu<\min\{4,N\}$ then 
\begin{align}\label{nh22}
b(u_1+t g_\rho^{\e,\sigma})	& 
\geq b(u_1)+b(t g_\rho^{\e,\sigma}) +\widehat{C} t^{2.2^*_{\mu}-1} \int_{\Om}\int_{ \Om}\frac{(g_\rho^{\e,\sigma}(x))^{2^*_{\mu}}(g_\rho^{\e,\sigma}(y))^{2^*_{\mu}-1}u_1(y)}{|x-y|^{\mu}}~dxdy\\& \quad + 2.2^*_{\mu} t \int_{\Om}\int_{ \Om}\frac{(u_1(x))^{2^*_{\mu}}(u_1(y))^{2^*_{\mu}-1}g_\rho^{\e,\sigma}(y)}{|x-y|^{\mu}}~dxdy- O(\e^{(\frac{2N-\mu}{4}) \Theta}) \text{ for all } \Theta <1, \nonumber
\end{align}
where $u_1$ is the local minimum obtained in Lemma \ref{nhlem10}.
\end{Lemma}
\begin{proof}
	We will divide the proof in two cases:\\
	\textbf{Case 1:} $ 2^{*}_{\mu} >3$.\\
	It is easy to see that  there exists $\widehat{A}>0$ such that 
	\begin{align*}
	(a+b)^p\geq a^p+b^p+pa^{p-1}b+\widehat{A}ab^{p-1} \text{ for all } a,\;b \geq 0 \text{ and } p>3, 
	\end{align*}
	which  implies that 
	\begin{align*}
	b(u_1+t g_\rho^{\e,\sigma})	& \geq b(u_1)+b(t g_\rho^{\e,\sigma}) +\widehat{C}t^{2.2^*_{\mu}-1} \int_{\Om}\int_{ \Om}\frac{(g_\rho^{\e,\sigma}(x))^{2^*_{\mu}}(g_\rho^{\e,\sigma}(y))^{2^*_{\mu}-1}u_1(y)}{|x-y|^{\mu}}~dxdy\\& \quad + 2.2^*_{\mu}t  \int_{\Om}\int_{ \Om}\frac{(u_1(x))^{2^*_{\mu}}(u_1(y))^{2^*_{\mu}-1}g_\rho^{\e,\sigma}(y)}{|x-y|^{\mu}}~dxdy,  \text{ where } \widehat{C}= \min\{\widehat{A}, 2.2^*_\mu\}.
	\end{align*}
	\textbf{Case 2:} $2< 2^{*}_{\mu} \leq 3$.\\
	We recall the inequality from Lemma 4 of \cite{brezis1}: there exist $ C$(depending on $2^*_{\mu}$) such that, for all $a,b\geq 0$, 
	\begin{align}\label{nh21}
	(a+b)^{2^*_{\mu}}\geq   \left\{
	\begin{array}{ll}
	a^{2^*_{\mu}}+b^{2^*_{\mu}}+2^*_{\mu}a^{2^*_{\mu}-1}b+ 2^*_{\mu}ab^{2^*_{\mu}-1} -C ab^{2^*_{\mu}-1} &  \text{ if } a\geq b , \\
	a^{2^*_{\mu}}+b^{2^*_{\mu}}+2^*_{\mu}a^{2^*_{\mu}-1}b+ 2^*_{\mu}ab^{2^*_{\mu}-1} -C a^{2^*_{\mu}-1}b &  \text{ if } a\leq b, \\
	\end{array} 
	\right. \end{align}	
	Consider $\Om \times \Om= O_1 \cup O_2 \cup O_3 \cup O_4$, where
	\begin{align*}
	& O_1= \{(x,y) \in \Om\times \Om  \mid u_1(x)\geq tg_\rho^{\e,\sigma}(x) \text{ and } u_1(y)\geq tg_\rho^{\e,\sigma}(y) \},\\
	& O_2= \{(x,y) \in \Om\times \Om  \mid u_1(x)\geq tg_\rho^{\e,\sigma}(x) \text{ and } u_1(y)< tg_\rho^{\e,\sigma}(y) \},\\
	& O_3= \{(x,y) \in \Om\times \Om  \mid u_1(x)< tg_\rho^{\e,\sigma}(x) \text{ and } u_1(y)\geq tg_\rho^{\e,\sigma}(y) \}, \\
	& O_4= \{(x,y) \in \Om\times \Om  \mid u_1(x)< tg_\rho^{\e,\sigma}(x) \text{ and } u_1(y)< tg_\rho^{\e,\sigma}(y) \}.
	\end{align*} 
	Also, define the $b(u)_{|O_i}= \ds \int\int_{O_i}\frac{(u(x))^{2^{*}_{\mu}}(u(y))^{2^{*}_{\mu}}}{|x-y|^{\mu}}~dxdy $,  for all $u \in H_0^1(\Om)$ and $i=1,2,3,4$.
	\textbf{Subcase 1:} when $(x,y) \in O_1$. \\
	Employing \eqref{nh21}, we have the following inequality:
	\begin{align*}
	b(u_1+t g_\rho^{\e,\sigma})_{|O_1} & \geq  (b(u_1)+b(t g_\rho^{\e,\sigma}))_{|O_1} +2.2^*_{\mu} t^{2.2^*_{\mu}-1} \int\int_{O_1}\frac{(g_\rho^{\e,\sigma}(x))^{2^*_{\mu}}(g_\rho^{\e,\sigma}(y))^{2^*_{\mu}-1}u_1(y)}{|x-y|^{\mu}}dxdy\\& \quad + 2.2^*_{\mu} t \iint_{O_1} \frac{(u_1(x))^{2^*_{\mu}}(u_1(y))^{2^*_{\mu}-1}g_\rho^{\e,\sigma}(y)}{|x-y|^{\mu}}dxdy - A^1_\e,
	\end{align*}
	where $A^1_\e$ is sum of eight non-negative integrals and  each integral  has an  upper bound of the form  $ C \int\int_{O_1}\frac{u_1(x)(tg_\rho^{\e,\sigma}(x))^{2^*_{\mu}-1} (u_1(y))^{2^*_{\mu}}}{|x-y|^{\mu}}~dxdy$ or $ C \int\int_{O_1}\frac{u_1(y)(tg_\rho^{\e,\sigma}(y))^{2^*_{\mu}-1} (u_1(x))^{2^*_{\mu}}}{|x-y|^{\mu}}~dxdy$.
	Write $(tg_\rho^{\e,\sigma}(x))^{2^*_{\mu}-1}= (tg_\rho^{\e,\sigma}(x))^{r}. (tg_\rho^{\e,\sigma}(x))^{s}$ with $2^*_{\mu}-1= r+s,\; 0<s<\frac{2^*_{\mu}}{2}$. Then 
	utilizing the definition of $O_1,\; u_1\in L^{\infty}(\Om)$ and Hardy-Littlewood-Sobolev inequality,  we have  
	\begin{equation*}
	\begin{aligned}
	\ds \int\int_{O_1} \frac{u_1(x)(tg_\rho^{\e,\sigma}(x))^{2^*_{\mu}-1} (u_1(y))^{2^*_{\mu}}}{|x-y|^{\mu}}~dxdy  & \leq C \int\int_{O_1}\frac{(u_1(x))^{1+r}(tg_\rho^{\e,\sigma}(x))^{s} (u_1(y))^{2^*_{\mu}}}{|x-y|^{\mu}}~dxdy  \\
	&\leq  C \int_{\Om}\int_{\Om}\frac{(tg_\rho^{\e,\sigma}(x))^{s} (u_1(y))^{2^*_{\mu}}}{|x-y|^{\mu}}~dxdy \\
	&\leq  C \int_{\Om}\int_{\Om}\frac{  \e^{\frac{s(N-2)}{2}}   }{|x-y|^{\mu} |x-(1-\e)\sigma|^{s(N-2)}}~dxdy \\
	& \leq C \e^{\frac{s(N-2)}{2}} \left( \ds \int_{\Om}
	\frac{dx}{|x-(1-\e)\sigma|^{\frac{s(2N)(N-2)}{2N-\mu}}}
	\right)^{\frac{2N-\mu}{2N}}  \\
	& \leq C \e^{\frac{s(N-2)}{2}} \left( \ds \int_{\Om}
	\frac{dx}{|x-(1-\e)\sigma|^{\frac{s(2N)(N-2)}{2N-\mu}}}
	\right)^{\frac{2N-\mu}{2N}}.
	\end{aligned}
	\end{equation*}
	By the choice of $s$ we have  $\ds \int_{\Om}
	\frac{dx}{|x-(1-\e)\sigma|^{\frac{s(2N)(N-2)}{2N-\mu}}} < \infty$.As a result, we get
	\begin{align*}
	\ds \int\int_{O_1}\frac{u_1(x)(tg_\rho^{\e,\sigma}(x))^{2^*_{\mu}-1} (u_1(y))^{2^*_{\mu}}}{|x-y|^{\mu}}~dxdy \leq O(\e^{(\frac{2N-\mu}{4}) \Theta}) \text{ for all } \Theta <1.
	\end{align*}	 
	In a similar manner, we have  
	\begin{align*}
	\ds C \int\int_{O_1}\frac{u_1(y)(tg_\rho^{\e,\sigma}(y))^{2^*_{\mu}-1} (u_1(x))^{2^*_{\mu}}}{|x-y|^{\mu}}~dxdy \leq O(\e^{(\frac{2N-\mu}{4}) \Theta}) \text{ for all } \Theta <1.
	\end{align*}	 
	\noi  \textbf{Subcase 2:} when $(x,y) \in O_2$. \\
	Once again using  \eqref{nh21}, we have the following inequality: 
	\begin{align*}
	b(u_1+t g_\rho^{\e,\sigma})_{|O_2} & \geq  [b(u_1)+b(t g_\rho^{\e,\sigma})]|_{O_2} +2.2^*_{\mu} t^{2.2^*_{\mu}-1} \iint_{O_2}\frac{(g_\rho^{\e,\sigma}(x))^{2^*_{\mu}}(g_\rho^{\e,\sigma}(y))^{2^*_{\mu}-1}u_1(y)}{|x-y|^{\mu}}dxdy\\& \quad + 2.2^*_{\mu} t \iint_{O_2} \frac{(u_1(x))^{2^*_{\mu}}(u_1(y))^{2^*_{\mu}-1}g_\rho^{\e,\sigma}(y)}{|x-y|^{\mu}}~dxdy - A^2_\e,
	\end{align*}
	where $A^2_\e$ is sum of eight non-negative integrals and  each integral  has an  upper bound of the form  $ C \int\int_{O_2}\frac{u_1(x)(tg_\rho^{\e,\sigma}(x))^{2^*_{\mu}-1} (g_\rho^{\e,\sigma}(y))^{2^*_{\mu}}}{|x-y|^{\mu}}~dxdy$ or $ C \int\int_{O_2}\frac{(u_1(y))^{2^*_{\mu}-1}(tg_\rho^{\e,\sigma}(y))(u_1(x))^{2^*_{\mu}}}{|x-y|^{\mu}}~dxdy$.
	By the similar estimates as in Subcase 1 and using the definition of $O_2$, the fact that $tg_\rho^{\e,\sigma}\in H_0^1(\Om)$ and  regularity of $u_1$,  we have  
	\begin{align*}
	\int\int_{O_2}\frac{u_1(x)(tg_\rho^{\e,\sigma}(x))^{2^*_{\mu}-1} (g_\rho^{\e,\sigma}(y))^{2^*_{\mu}}}{|x-y|^{\mu}}~dxdy \leq O(\e^{(\frac{2N-\mu}{4}) \Theta}) \text{ for all } \Theta <1.
	\end{align*}
Write $(u_1(y))^{2^*_{\mu}-1}= (u_1(y))^{r}. (u_1(y))^{s}$ with $2^*_{\mu}-1= r+s,\; 0<1+s<\frac{2^*_{\mu}}{2}$. Then 
		utilizing the definition of $O_2$, $u_1\in L^{\infty}(\Om)$ and Hardy-Littlewood-Sobolev inequality,  we have  
	\begin{equation*}
	\begin{aligned}
	\int\int_{O_2} \frac{(u_1(y))^{2^*_{\mu}-1}(tg_\rho^{\e,\sigma}(y))(u_1(x))^{2^*_{\mu}}}{|x-y|^{\mu}}&~dxdy
 	\leq  \int\int_{O_2}\frac{(u_1(y))^{r}(tg_\rho^{\e,\sigma}(y))^{1+s}(u_1(x))^{2^*_{\mu}}}{|x-y|^{\mu}}~dxdy\\
	&\leq  C \int_{\Om}\int_{\Om}\frac{(tg_\rho^{\e,\sigma}(y))^{1+s}(u_1(x))^{2^*_{\mu}}}{|x-y|^{\mu}}~dxdy\\
	&\leq  C \int_{\Om}\int_{\Om}\frac{  \e^{\frac{(1+s)(N-2)}{2}}   }{|x-y|^{\mu} |y-(1-\e)\sigma|^{(1+s)(N-2)}}~dxdy \\
	&\leq  C  \e^{\frac{(1+s)(N-2)}{2}} \left( \ds \int_{\Om}
	\frac{dy}{|y-(1-\e)\sigma|^{\frac{(1+s)(2N)(N-2)}{2N-\mu}}}
	\right)^{\frac{2N-\mu}{2N}} \\
	& \leq C \e^{\frac{(1+s)(N-2)}{2}} \left( \ds \int_{\Om}
	\frac{dy}{|y-(1-\e)\sigma|^{\frac{(1+s)(2N)(N-2)}{2N-\mu}}}
	\right)^{\frac{2N-\mu}{2N}}.
	\end{aligned}
	\end{equation*}
	By the choice of  $s$ we have  $ \ds \int_{\Om}
	\frac{dx}{|x-(1-\e)\sigma|^{\frac{(1+s)(2N)(N-2)}{2N-\mu}}} < \infty$. Hence we obtain 
	\begin{align*}
	\int\int_{O_2}\frac{(u_1(y))^{2^*_{\mu}-1}(tg_\rho^{\e,\sigma}(y))(u_1(x))^{2^*_{\mu}}}{|x-y|^{\mu}}~dxdy \leq O(\e^{(\frac{2N-\mu}{4}) \Theta}) \text{ for all } \Theta <1.
	\end{align*}	 
	\noi  \textbf{Subcase 3:} when $(x,y) \in O_3$. \\
	Using \eqref{nh21}, we have 
	\begin{align*}
	b(u_1+t g_\rho^{\e,\sigma})_{|O_3} & \geq  (b(u_1)+b(t g_\rho^{\e,\sigma}))|_{O_3} +2.2^*_{\mu} t^{2.2^*_{\mu}-1} \iint_{O_3}\frac{(g_\rho^{\e,\sigma}(x))^{2^*_{\mu}}(g_\rho^{\e,\sigma}(y))^{2^*_{\mu}-1}u_1(y)}{|x-y|^{\mu}}dxdy\\& \quad + 2.2^*_{\mu} t \iint_{O_3} \frac{(u_1(x))^{2^*_{\mu}}(u_1(y))^{2^*_{\mu}-1}g_\rho^{\e,\sigma}(y)}{|x-y|^{\mu}}~dxdy - A^3_\e,
	\end{align*}
	where $A^3_\e$ is sum of eight non-negative integrals and  each integral  has an  upper bound of the form $ C \iint_{O_3}\frac{(u_1(x))^{2^*_{\mu}-1}(tg_\rho^{\e,\sigma}(x)) (u_1(y))^{2^*_{\mu}}}{|x-y|^{\mu}}~dxdy$ or  $ C \iint_{O_3}\frac{u_1(y)(tg_\rho^{\e,\sigma}(y))^{2^*_{\mu}-1} (g_\rho^{\e,\sigma}(x))^{2^*_{\mu}}}{|x-y|^{\mu}}~dxdy$.
	By the similar estimates as in Subcase 1, Subcase 2 and using the definition of $O_3$ and  regularity of $u_1$,  we get  $A^3_\e \leq O(\e^{(\frac{2N-\mu}{4}) \Theta}) \text{ for all } \Theta <1$. \\
	\noi  \textbf{Subcase 4:} when $(x,y) \in O_4$. \\
	Using \eqref{nh21}, we have 
	\begin{align*}
	b(u_1+t g_\rho^{\e,\sigma})_{|O_4} & \geq  (b(u_1)+b(t g_\rho^{\e,\sigma}))|_{O_4} +2.2^*_{\mu} t^{2.2^*_{\mu}-1} \iint_{O_4}\frac{(g_\rho^{\e,\sigma}(x))^{2^*_{\mu}}(g_\rho^{\e,\sigma}(y))^{2^*_{\mu}-1}u_1(y)}{|x-y|^{\mu}}dxdy\\& \quad + 2.2^*_{\mu} t \iint_{O_4} \frac{(u_1(x))^{2^*_{\mu}}(u_1(y))^{2^*_{\mu}-1}g_\rho^{\e,\sigma}(y)}{|x-y|^{\mu}}~dxdy - A^4_\e,
	\end{align*}
	where $A^4_\e$ is sum of eight non-negative integrals and  each integral  has an  upper bound of the form $ C \int\int_{O_4}\frac{(u_1(x))^{2^*_{\mu}-1}(tg_\rho^{\e,\sigma}(x)) (tg_\rho^{\e,\sigma}(y))^{2^*_{\mu}}}{|x-y|^{\mu}}~dxdy$ or  $ C \int\int_{O_4}\frac{u_1(y)(tg_\rho^{\e,\sigma}(y))^{2^*_{\mu}-1} (g_\rho^{\e,\sigma}(x))^{2^*_{\mu}}}{|x-y|^{\mu}}~dxdy$.
	By the similar estimates as in Subcase 2, we have  
	\begin{align*}
	A^4_\e \leq O(\e^{(\frac{2N-\mu}{4}) \Theta}) \text{ for all } \Theta <1.
	\end{align*}
	From all subcases we obtain $ A^i_\e \leq O(\e^{(\frac{2N-\mu}{4}) \Theta}) \text{ for all } \Theta <1 \text{ and } i=1,2,3,4.$
	Combining all sub cases we conclude Case 2. From Case 1 and Case 2 we have the required result. \QED
	\end{proof}

\begin{Proposition}\label{nhprop12}
Let $\mu<\min\{ 4,\; N \}$ then there exists $\e_0>0$ such that for every $0<\e<\e_0$ we have 
\begin{align*}
\sup_{t\geq 0}\mathcal{J}_f(u_1+t g_\rho^{\e,\sigma})< \mathcal{J}_f(u_1)+	\frac{N-\mu+2}{2(2N-\mu)}S_{H,L}^{\frac{2N-\mu}{N-\mu+2}} \text{ uniformly in } \sigma \in \mathbb{S}^{N-1},
\end{align*}
where $u_1$ is the local minimum in Lemma \ref{nhlem10}.
\end{Proposition}
\begin{proof}
Observe the fact that  by Lemma \ref{nhlem34} we have $u \in L^{\infty}(\Om)$. Moreover, $u>0$  and $g_\rho^{\e,\sigma}\geq 0$ in $\Om$. This implies $b(u_1+t g_\rho^{\e,\sigma}) = \ds \int_{\Om}\int_{\Om}\frac{(u_1+tg_\rho^{\e,\sigma}(x))^{2^{*}_{\mu}}(u_1+tg_\rho^{\e,\sigma}(y))^{2^{*}_{\mu}}}{|x-y|^{\mu}}~dxdy$\\
\textbf{ Claim 1:} There exists a $R_0>0$ such that 
	\begin{align*}
I= \int_{\Om}\int_{\Om}\frac{(g_\rho^{\e,\sigma}(x))^{2^{*}_{\mu}}(g_\rho^{\e,\sigma}(y))^{2^{*}_{\mu}-1}u_1(y)}{|x-y|^{\mu}}~dxdy\geq \widehat{C}R_0\e^{\frac{N-2}{2}}.
\end{align*} 
Clearly,
\begin{align*}
I& \geq \ds \int_{B_{\frac{1}{2\rho}}\setminus B_{2\rho}}\int_{B_{\frac{1}{2\rho}}\setminus B_{2\rho}}\frac{(g_\rho^{\e,\sigma}(x))^{2^{*}_{\mu}}(g_\rho^{\e,\sigma}(y))^{2^{*}_{\mu}-1}u_1(y)}{|x-y|^{\mu}}~dxdy\\
& \geq C  \ds \int_{B_{\frac{1}{2\rho}}\setminus B_{2\rho}}\int_{B_{\frac{1}{2\rho}}\setminus B_{2\rho}}\frac{(u_{\e}^{\sigma}(x))^{2^{*}_{\mu}}(u_{\e}^{\sigma}(y))^{2^{*}_{\mu}-1}}{|x-y|^{\mu}}~dxdy\\
&\geq  C \ds \int_{B_{\frac{1}{2\rho}}\setminus B_{2\rho}}\int_{B_{\frac{1}{2\rho}}\setminus B_{2\rho}}\frac{\e^{\frac{3N}{2}+1-\mu}~dxdy}{|x-y|^{\mu} (\e^2+|x-(1-\e)\sigma|^2)^{\frac{2N-\mu}{2}}(\e^2+|y-(1-\e)\sigma|^2)^{\frac{N-\mu+2}{2}}}.
\end{align*}
For any $\e < 1-2\rho$ there exists $c>0$ such that $1-\e>c>2\rho$ so we get 
\begin{align*}
I&\geq  C \e^{\frac{3N}{2}+1-\mu}\ds \int_{B_{c}}\int_{B_{c}}\frac{dzdw}{|z-w|^{\mu} (\e^2+|z|^2)^{\frac{2N-\mu}{2}}(\e^2+|w|^2)^{\frac{N-\mu+2}{2}}}\\
&\geq  C \e^{\frac{N-2}{2}}\ds \int_{B_{c}}\int_{B_{c}}\frac{dzdw}{|z-w|^{\mu} (1+|z|^2)^{\frac{2N-\mu}{2}}(1+|w|^2)^{\frac{N-\mu+2}{2}}}= O(\e^{\frac{N-2}{2}}).
\end{align*}
This proves the claim 1. Now using Lemma \ref{nhlem33}, we have
\begin{align*}
\mathcal{J}_f(u_1+t g_\rho^{\e,\sigma})  \leq & \frac{1}{2}a(u_1)+\frac{1}{2}a(t g_\rho^{\e,\sigma}) +t\ld u_1, g_\rho^{\e,\sigma} \rd_{H_0^1(\Om)}-  \frac{1}{2.2^*_{\mu}} b(u_1)-\frac{1}{2.2^*_{\mu}} b(t g_\rho^{\e,\sigma})\\&   - \widehat{C} t^{2.2^*_{\mu}-1} \int_{\Om}\int_{ \Om}\frac{(g_\rho^{\e,\sigma}(x))^{2^*_{\mu}}(g_\rho^{\e,\sigma}(y))^{2^*_{\mu}-1}u_1(y)}{|x-y|^{\mu}}~dxdy - \int_{ \Om}fu_1~dx \\&- t\int_{ \Om}fg_\rho^{\e,\sigma} ~dx
  - t\int_{\Om}\int_{ \Om}\frac{(u_1(x))^{2^*_{\mu}}(u_1(y))^{2^*_{\mu}-1}g_\rho^{\e,\sigma}(y)}{|x-y|^{\mu}}~dxdy+ O(\e^{(\frac{2N-\mu}{4}) \Theta}).
\end{align*}
 for all  $\Theta <1.$ Taking $\Theta = \frac{2}{2^*_{\mu}}$, we have 
\begin{align*}
\mathcal{J}_f(u_1+t g_\rho^{\e,\sigma})  \leq & \frac{1}{2}a(u_1)+\frac{1}{2}a(t g_\rho^{\e,\sigma}) +t\ld u_1, g_\rho^{\e,\sigma} \rd_{H_0^1(\Om)}-  \frac{1}{2.2^*_{\mu}} b(u_1)-\frac{1}{2.2^*_{\mu}} b(t g_\rho^{\e,\sigma})\\&   -  \widehat{C} t^{2.2^*_{\mu}-1} \int_{\Om}\int_{ \Om}\frac{(g_\rho^{\e,\sigma}(x))^{2^*_{\mu}}(g_\rho^{\e,\sigma}(y))^{2^*_{\mu}-1}u_1(y)}{|x-y|^{\mu}}~dxdy - \int_{ \Om}fu_1~dx \\& - t\int_{ \Om}fg_\rho^{\e,\sigma} ~dx
  - t\int_{\Om}\int_{ \Om}\frac{(u_1(x))^{2^*_{\mu}}(u_1(y))^{2^*_{\mu}-1}g_\rho^{\e,\sigma}(y)}{|x-y|^{\mu}}~dxdy+ o(\e^{\frac{N-2}{2} }) .
\end{align*}
This on utilizing Lemma \ref{nhlem9} and  claim 1 gives  
\begin{align*}
\mathcal{J}_f(u_1+t g_\rho^{\e,\sigma}) & \leq \frac{1}{2}a(u_1)+\frac{1}{2}a(t g_\rho^{\e,\sigma}) +t\ld u_1, g_\rho^{\e,\sigma} \rd_{H_0^1(\Om)}-  \frac{1}{2.2^*_{\mu}} b(u_1)-\frac{1}{2.2^*_{\mu}} b(t g_\rho^{\e,\sigma})\\&  \quad -\widehat{C} t^{2.2^*_{\mu}-1} \int_{\Om}\int_{ \Om}\frac{(g_\rho^{\e,\sigma}(x))^{2^*_{\mu}}(g_\rho^{\e,\sigma}(y))^{2^*_{\mu}-1}}{|x-y|^{\mu}}~dxdy - \int_{ \Om}fu_1~dx- t\int_{ \Om}fg_\rho^{\e,\sigma} ~dx
\\&\quad  - t\int_{\Om}\int_{ \Om}\frac{(u_1(x))^{2^*_{\mu}}(u_1(y))^{2^*_{\mu}-1}g_\rho^{\e,\sigma}(y)}{|x-y|^{\mu}}~dxdy + o(\e^{\frac{N-2}{2} }) 
\\
& = \mathcal{J}_f(u_1)+\mathcal{J}(tg_\rho^{\e,\sigma})-\widehat{C} t^{2.2^*_{\mu}-1}  \int_{\Om}\int_{ \Om}\frac{(g_\rho^{\e,\sigma}(x))^{2^*_{\mu}}(g_\rho^{\e,\sigma}(y))^{2^*_{\mu}-1}}{|x-y|^{\mu}}~dxdy + o(\e^{\frac{N-2}{2} }) \\
& \leq  \mathcal{J}_f(u_1)+ \frac{t^2}{2}\left(S_{H,L}^{\frac{2N-\mu}{N-\mu+2}}+O(\e^{N-2})\right)-\frac{t^{2.2^*_{\mu}}}{2.2^*_{\mu}} \left(S_{H,L}^{\frac{2N-\mu}{N-\mu+2}}-O(\e^{\frac{2N-\mu}{2}})\right)\\
& \quad - t^{2.2^*_{\mu}-1} \widehat{C}R_0\e^{\frac{N-2}{2}} + o(\e^{\frac{N-2}{2} }) .
\end{align*}
Now define $K(t):= \ds \frac{t^2}{2}\left(S_{H,L}^{\frac{2N-\mu}{N-\mu+2}}+O(\e^{N-2})\right)-\frac{t^{2.2^*_{\mu}}}{2.2^*_{\mu}} \left(S_{H,L}^{\frac{2N-\mu}{N-\mu+2}}-O(\e^{\frac{2N-\mu}{2}})\right)- t^{2.2^*_{\mu}-1}\widehat{C}R_0\e^{\frac{N-2}{2}}$ 
then $K(t)\ra \infty$ as $t\ra \infty$ and $\lim_{t\ra 0^+}K(t)>0$ so there exists a $t_\e>0$ such that $\ds \sup_{t>0}K(t)$ is attained and $t_\e< \left(\frac{S_{H,L}^{\frac{2N-\mu}{N-\mu+2}}+O(\e^{N-2})}{S_{H,L}^{\frac{2N-\mu}{N-\mu+2}}-O(\e^{\frac{2N-\mu}{2}})}\right)^{\frac{1}{2.2^*_{\mu}-2}}:=S_{H,L}(\e)$. Moreover there exists a $t_1>0$ such that for sufficiently small $\e>0$ we have $t_\e>t_1$. Clearly the function 
\begin{align*}
t \mapsto \ds \frac{t^2}{2}\left(S_{H,L}^{\frac{2N-\mu}{N-\mu+2}}+O(\e^{N-2})\right)-\frac{t^{2.2^*_{\mu}}}{2.2^*_{\mu}} \left(S_{H,L}^{\frac{2N-\mu}{N-\mu+2}}-O(\e^{\frac{2N-\mu}{2}})\right)
\end{align*}
is an increasing function in $[0,S_{H,L}(\e)]$. Therefore, 
\begin{align*}
\sup_{t\geq 0} \mathcal{J}_f(u_1+t g_\rho^{\e,\sigma})\leq & \mathcal{J}_f(u_1)+  	\frac{N-\mu+2}{2(2N-\mu)}S_{H,L}^{\frac{2N-\mu}{N-\mu+2}}+\ds O(\e^{\min\{\frac{2N-\mu}{2},\; N-2\}})\\ & - t_1^{2.2^*_{\mu}-1}\widehat{C}R_0\e^{\frac{N-2}{2}}  + o(\e^{\frac{N-2}{2} }) . 
\end{align*}
Hence there exits a $\e_0>0$ such that for every $0<\e<\e_0$ we have 
\begin{align*}
\sup_{t\geq 0}\mathcal{J}_f(u_1+t g_\rho^{\e,\sigma})< \mathcal{J}_f(u_1)+	\frac{N-\mu+2}{2(2N-\mu)}S_{H,L}^{\frac{2N-\mu}{N-\mu+2}} \text{ uniformly in } \sigma \in \mathbb{S}^{N-1}.
\end{align*}
\QED
\end{proof}	
\begin{Lemma}\label{nhlem13}
	The following holds:
	\begin{itemize}
		\item [(i)] $H_0^1(\Om)\setminus \mathcal{N}_f^-= U_1\cup U_2$,
		where 
		\begin{align*}
		& U_1:= \left \lbrace u \in H_0^1(\Om)\setminus \{0\}\; \middle|\;  u^+ \not \equiv 0,\; \|u\|<t^-\left(\frac{u}{\|u\|}\right)  \right \rbrace \cup \{0\},\\
		& U_2:= \left \lbrace u \in H_0^1(\Om)\setminus \{0\}\; \middle|\;  u^+ \not \equiv 0,\; \|u\|>t^-\left(\frac{u}{\|u\|}\right)  \right \rbrace. 
		\end{align*}
		\item [(ii)] $\mathcal{N}_f^+ \subset U_1$.
		\item [(iii)] For each $0<\e\leq \e_0$, there exists $t_0>1$ and such that $u_1+t_0g_\rho^{\e,\sigma}\in U_2$.
		\item [(iv)] For each $0<\e<\e_0$, there exists $s_0\subset (0,1)$ and such that $u_1+ s_0 t_0g_\rho^{\e,\sigma}\in \mathcal{N}_f^-$.
		\item [(v)] $\Upsilon_f^- < \Upsilon_f +	\frac{N-\mu+2}{2(2N-\mu)}S_{H,L}^{\frac{2N-\mu}{N-\mu+2}}$.
	\end{itemize}
\end{Lemma}
\begin{proof}
	\begin{itemize}
		\item [(i)] It holds by Lemma \ref{nhlem3} (d).
		\item[(ii)] Let $u \in \mathcal{N}_f^+$ then $t^+(u)=1$ and  $1<t^+(u)<t_{max}<t^-(u)= \frac{1}{\|u\|}t^-\left(\frac{u}{\|u\|}\right)$ that is,  $\mathcal{N}_f^+ \subset U_1$.
		\item [(iii)] First, we will show that there exists a constant $c>0$ such that $0<t^-\left(\frac{u_1+tg_\rho^{\e,\sigma}}{\|u_1+tg_\rho^{\e,\sigma}\|}\right)< c $ for all $t>0$. On the contrary,  let there exist a sequence $\{t_n\}$ such that $t_n\ra \infty$ and $t^-\left(\frac{u_1+t_ng_\rho^{\e,\sigma}}{\|u_1+t_ng_\rho^{\e,\sigma}\|}\right)\ra \infty$ as $n\ra \infty$. Define $u_n:= \frac{u_1+t_ng_\rho^{\e,\sigma}}{\|u_1+t_ng_\rho^{\e,\sigma}\|}$ so there exists $t^-(u_n)$ such that  $t^-(u_n)u_n \in \mathcal{N}_f^-$. By dominated convergence theorem,
		\begin{align*}
		b(u_n)= \frac{b(u_1+t_ng_\rho^{\e,\sigma})}{\|u_1+t_ng_\rho^{\e,\sigma}\|^{2.2^*_{\mu}}}= \frac{b(\frac{u_1}{t_n} +g_\rho^{\e,\sigma})}{\|\frac{u_1}{t_n} +g_\rho^{\e,\sigma}\|^{2.2^*_{\mu}}}\ra \frac{b(g_\rho^{\e,\sigma})}{\|g_\rho^{\e,\sigma}\|^{2.2^*_{\mu}}} \text{ as } n\ra \infty.
		\end{align*}
		Hence, $\mathcal{J}_f(t^-(u_n)u_n)\ra -\infty$ as $n\ra \infty$, contradicts the fact that $\mathcal{J}_f$ is bounded below on $\mathcal{N}_f$. Therefore,  there exists $c>0$ such that $0<t^-\left(\frac{u_1+tg_\rho^{\e,\sigma}}{\|u_1+tg_\rho^{\e,\sigma}\|}\right)< c $ for all $t>0$. Let $\ds t_0= \frac{|c^2-\|u_1\|^2|^{\frac{1}{2}}}{\|g_\rho^{\e,\sigma}\|}+1$ then
		\begin{align*}
		\|u_1+t_0g_\rho^{\e,\sigma}\|^2 & =\|u_1\|^2+t_0^2\|g_\rho^{\e,\sigma}\|^2+2t_0\ld u_1, \; g_\rho^{\e,\sigma} \rd \\
		& \geq \|u_1\|^2+ |c^2-\|u_1\|^2| \geq c^2\geq \left(t^-\left(\frac{u_1+tg_\rho^{\e,\sigma}}{\|u_1+tg_\rho^{\e,\sigma}\|}\right)\right)^2.
		\end{align*}
		It implies that $u_1+t_0g_\rho^{\e,\sigma}\in U_2$.
		\item [(iv)] For each $0<\e<\e_0$, define a path $\xi_{\e}(s)= u_1+ s t_0g_\rho^{\e,\sigma}$ for $s\in [0,1]$. Then 
		\begin{align*}
		\xi_{\e}(0)= u_1 \quad \text{and} \quad \xi_{\e}(1)= u_1+ t_0g_\rho^{\e,\sigma} \in U_2.
		\end{align*}
		Since $\frac{1}{\|u\|}t^-\left(\frac{u}{\|u\|}\right)$ is a continuous function and $\xi_{\e}([0,1])$ is connected. So, there exists $s_0 \in [0,1]$ such that $\xi_{\e}(s_0)=u_1+ s_0t_0g_\rho^{\e,\sigma} \in \mathcal{N}_f^- $.
		\item [(v)] Using part (d) and Proposition \ref{nhprop12}.\QED
	\end{itemize}
	\end{proof}
 At this point  we will state Global compactness Lemma for the functional $\mathcal{J}_f$ which is a version of Theorem 4.4 of \cite{choqcoron}.  
\begin{Lemma}\label{global}
	Let  $\{u_n\}\subset H_0^1(\Om)$ be such that $\mathcal{J}_f(u_n)\ra c,\;  \mathcal{J}_f^{\prime}(u_n)\ra 0 $.
	Then passing if necessary to a subsequence, there exists a solution $v_0 \in H_0^1(\Om) $ of 
	\[
	-\Delta u= \left(\int_{\Om}\frac{|u^+(y)|^{2^*_{\mu}}}{|x-y|^{\mu}}dy\right)
	|u^+|^{2^*_{\mu}-1} +f \text{ in } \Om 
	\]
	and (possibly) $k\in \mathbb{N}\cup \{0\}$, non-trivial solutions $\{v_1,v_2,...,v_k\}$ of 
	\[
	-\Delta u= (|x|^{-\mu}*|u^+|^{2^*_{\mu}})|u^+|^{2^*_{\mu}-1} \text{ in } \mathbb{R}^N
	\]
	with $v_i \in D^{1,2}(\mathbb{R}^N) $ and $k$ sequences $\{y_n^i\}_n \subset \Om $ and $\{\la_n^i\}_n \subset \mathbb{R}_+$  $i=1,2,\cdots k$, satisfying 
	\begin{equation*}
	\begin{aligned}
	 \frac{1}{{\la}_n^i} dist &(y_n^i, \partial \Om) \ra \infty ,\; \text{and} \; \|u_n-v_0-\sum_{i=1}^{k}(\la_n^i)^{\frac{2-N}{2}}v_i((.-y_n^i)/\la_n^i)\|_{D^{1,2}(\mathbb{R}^N)} \ra 0 ,\; n \ra \infty,\\	
	& \|u_n\|^2_{D^{1,2}(\mathbb{R}^N)}\ra \sum_{i=0}^{k}\|v_i\|^2_{D^{1,2}(\mathbb{R}^N)},  \text{ as } n\ra \infty,\;\;\;
	\mathcal{J}_f(v_0)+\sum_{i=1}^{k} \mathcal{J}_{\infty}(v_i)=c,
	\end{aligned}
	\end{equation*}
	where $\mathcal{J}_{\infty}(u):= \frac12 \ds \int_{\mathbb{R}^N} |\nabla u|^2~dx- \frac{1}{2.2^*_{\mu}}  \int_{\mathbb{R}^N}\int_{\mathbb{R}^N}\frac{|u^+(x)|^{2^{*}_{\mu}}|u^+(y)|^{2^{*}_{\mu}}}{|x-y|^{\mu}}~dxdy, \quad  u \in  D^{1,2}(\mathbb{R}^N)$.
\end{Lemma}

\begin{Lemma}\label{nhlem18}
	\begin{enumerate}
		\item[(i)] Let $\{u_n\}$ be a $(PS)_{c}$ sequence for $\mathcal{J}_f$ with $c<\Upsilon_f(\Om)+	\frac{N-\mu+2}{2(2N-\mu)}S_{H,L}^{\frac{2N-\mu}{N-\mu+2}}$ then there exists  a subsequence  still denoted by $\{u_n\}$ and a nonzero $u^0 \in H_0^1(\Om)$ such that $u_n\ra u^0$ strongly in $H_0^1(\Om)$ and $\mathcal{J}_f(u^0)=c$. 
		\item[(ii)]  Let $ \{u_n\}\subset \mathcal{N}_f^-$ be a $(PS)_{c}$ sequence for $\mathcal{J}_f$ with
		\begin{align*}
		\Upsilon_f(\Om)+	\frac{N-\mu+2}{2(2N-\mu)}S_{H,L}^{\frac{2N-\mu}{N-\mu+2}}< c< \Upsilon_f^-(\Om)+	\frac{N-\mu+2}{2(2N-\mu)}S_{H,L}^{\frac{2N-\mu}{N-\mu+2}}
		\end{align*}
		then there exists  subsequence  still denoted by $\{u_n\}$ and a nonzero $u^0 \in\mathcal{N}_f^-$ such that $u_n\ra u^0$ strongly in $H_0^1(\Om)$ and $\mathcal{J}_f(u^0)=c$. 
	\end{enumerate}
\end{Lemma}
\begin{proof}
	Proof of (i) follows from \cite[Lemma 3.4]{zampyang}.
	To prove (ii), Let  $\{u_n\}$ be a $(PS)_c$ sequence then  by standard arguments $\{u_n\}$ is bounded in $H_0^1(\Om)$ and there exists a subsequence of $\{u_n\}$ still denoted by $\{u_n\}$ and $u^0 \in H_0^1(\Om)$ such that $u_n\rp u^0$ in $H_0^1(\Om)$ and $\mathcal{J}_f^{\prime}(u^0)=0$. Then by Lemma \ref{nhlem6}, we have either $u^0 \in \mathcal{N}_f^-$ or $u^0 =u_1 $. Now using Lemma \ref{global} we obtain 
	\begin{align*}
	\Upsilon_f^-(\Om) +	\frac{N-\mu+2}{2(2N-\mu)}S_{H,L}^{\frac{2N-\mu}{N-\mu+2}}\geq c=   \mathcal{J}_f(u^0)+\sum_{i=1}^{k} \mathcal{J}_{\infty}(v_i)\geq \Upsilon_f(\Om) + k 	\frac{N-\mu+2}{2(2N-\mu)}S_{H,L}^{\frac{2N-\mu}{N-\mu+2}}.
	\end{align*}
	which on using Lemma \ref{nhlem13}(e), gives $k\leq 1$. By corollary 3.3 of \cite{choqcoron}, we get $v_1$ is a constant multiple of Talenti function that is,  $\mathcal{J}_{\infty}(v_1)= 	\frac{N-\mu+2}{2(2N-\mu)}S_{H,L}^{\frac{2N-\mu}{N-\mu+2}}$.
	If $k=0$ then we are done and if $k=1$ and $u^0=u_1$. Therefore,
	\begin{align*}
	c=  \mathcal{J}_f(u^0)+	\frac{N-\mu+2}{2(2N-\mu)}S_{H,L}^{\frac{2N-\mu}{N-\mu+2}}= \Upsilon_f(\Om)+ 	\frac{N-\mu+2}{2(2N-\mu)}S_{H,L}^{\frac{2N-\mu}{N-\mu+2}},
	\end{align*}
	a contradiction. If $k=1$ and  $u^0\in\mathcal{N}_f^-$, we get 
	\begin{align*}
	c =\mathcal{J}_f(u^0)+	\frac{N-\mu+2}{2(2N-\mu)}S_{H,L}^{\frac{2N-\mu}{N-\mu+2}}\geq  \Upsilon_f^-(\Om) + 	\frac{N-\mu+2}{2(2N-\mu)}S_{H,L}^{\frac{2N-\mu}{N-\mu+2}},
	\end{align*}
	which is again a contradiction. Hence $k=0$.\QED
\end{proof}

		\section{Existence of Second and third Solution}
	In this section we will show the existence of second and third solution of problem $(P_f)$. To prove this, we shall show that for a sufficiently small $\de>0$,
		\begin{align*}
		cat\bigg (\bigg \{ u \in \mathcal{N}_f^- \; : \; \mathcal{J}_f \leq \Upsilon_f(\Om)+	\frac{N-\mu+2}{2(2N-\mu)}S_{H,L}^{\frac{2N-\mu}{N-\mu+2}}- \de \bigg \} \bigg)\geq 2, 
		\end{align*}
	where $cat(X)$ is the category  of the  set $X$ is   defined in the Definition \ref{defi1}. And then employing Lemma \ref{nhlem31}, we conclude the existence of second and third solutions. We shall first gather some preliminaries. 	\\
	 For $c\in \mathbb{R}$, we define
		\begin{align*}
		b_c(u)= cb(u),\; \mathcal{I}_c(u)= \frac{1}{2}a(u)-\frac{1}{2.2^*_{\mu}}b_c(u),\; \mathcal{M}_c=\{u \in H_0^1(\Om)\setminus\{0\}| \ld \mathcal{I}_c^{\prime}(u),\; u\rd =0  \}.
		\end{align*}
		\noi	We denote 
		\begin{align}\label{nh20}
		[\mathcal{J}_f\leq c]= \{ u\in \mathcal{N}_f^-| \mathcal{J}_f(u) \leq c \}.
		\end{align}
		\begin{Definition} \label{defi1}
			\begin{enumerate}
				\item [(i)] For a topological space $X$, we say that a non-empty, closed subset $Y \subset X$ is contractible to a
				point in $X$ if and only if there exists a continuous mapping $ \xi :[0,1]\times Y\ra X$ such that for some $x_0 \in X$ $\xi(0,x)=x$ for all $x \in Y$, and $x_0 \in X$ $\xi(1,x)=x_0$ for all $x \in Y$.
				\item[(ii)] We define 
				\begin{align*}
				cat(X)= \min \lbrace  k & \in \mathbb{N} \; |  \text{ there exists closed subsets } Y_1,Y_2,\cdots Y_k \subset X \text{ such that } \\
				& Y_j \text{ is contractible to a point in } X \text{ for all } j \text{ and } \ds \cup_{j=1}^{k} Y_j =X	 \rbrace
				\end{align*}
			\end{enumerate}
		\end{Definition}
		\begin{Lemma}\label{nhlem31}\cite{ambrosetti}
			Suppose that $X$ is a Hilbert manifold and $G \in C^1(X,\mathbb{R})$ . Assume that there are $c_1 \in \mathbb{R}$ and  $k \in  \mathbb{N}$,
			such that
			\begin{enumerate}
				\item $G$ satisfies the Palais-Smale condition for energy level $c\leq c_1$;
				\item  $cat(\{x\in X\; |\; G(x)\leq c_1 \})\geq k$.
			\end{enumerate}
			Then $G$ has at least $k$ critical points in  $\{x\in X\; |\; G(x)\leq c_1 \}$.\QED
		\end{Lemma}
		\begin{Lemma}\label{nhlem32}\cite[Theorem 2.5] {adachifour}
			Let $X$ be a topological space. Suppose that there are two continuous maps $\Phi:\mathbb{S}^{N-1} \ra X$ and $ \Psi: X\ra \mathbb{S}^{N-1} $ such that $\Psi o \Phi$ is homotopic to the identity map of $\mathbb{S}^{N-1}$. Then $cat(X)\geq 2$.\QED
		\end{Lemma}
	Now we will proof a Lemma which will relate the functional $\mathcal{J}_f$ and $\mathcal{I}_c$. 	Note that for each $u \in H_0^1(\Om)$ there exists a unique $t^- >0  $ and a unique $t^*>0$ such that $t^-u \in \mathcal{N}_f^-$ and $t^* u \in \mathcal{N} $.
	
\begin{Lemma}\label{nhlem15}
	\begin{itemize}
		\item [(i)] For each $ u\in \Sigma:= \{u \in H_0^1(\Om)| \; \; \|u\|=1\}$, there exists a unique $t_c(u)>0$ such that $t_c(u)u\in \mathcal{M}_c$ and 
		\begin{align*}
		\max_{t\geq 0} \mathcal{I}_c(tu)= \mathcal{I}_c(t_c(u)u)= \frac{N-\mu+2}{2(2N-\mu)}\left( b_c(u)\right)^{-\frac{N-2}{N-\mu+2}}.
		\end{align*}
		\item[(ii)]  For each $u \in H_0^1(\Om)$ with $u^+\not \equiv 0$ and 		$0<\omega<1$, we have 
		\begin{align*}
		(1-\omega)\mathcal{I}_{\frac{1}{1-\omega}}(u)- \frac{1}{2\omega}\|f\|_{H^{-1}}^2 \leq \mathcal{J}_f(u)\leq (1+\omega)\mathcal{I}_{\frac{1}{1+\omega}}(u)+ \frac{1}{2\omega}\|f\|_{H^{-1}}^2
		\end{align*}
		\item[(iii)] For each $u \in \Sigma$ and	$0<\omega<1$, we have 
		\[
		(1-\omega)^{\frac{2N-\mu}{N-\mu+2}}\mathcal{J}(t^*u)- \frac{1}{2\omega}\|f\|_{H^{-1}}^2 \leq \mathcal{J}_f(t^- u)\leq (1+\omega)^{\frac{2N-\mu}{N-\mu+2}}\mathcal{J}(t^*u)+ \frac{1}{2\omega}\|f\|_{H^{-1}}^2.\]
	\item [(iv)] There exists $e_{11}>0$ such that if $0<\|f\|_{H^{-1}}< e_{11}$ then $\Upsilon_f^->0$.
	\end{itemize}
\end{Lemma}
	\begin{proof}
		\begin{itemize}
			\item [(i)]	For each $ u\in \Sigma$ define $k(t)= \ds \frac{1}{2}t^2- \frac{t^{2.2^*_{\mu}}}{2.2^*_{\mu}}b_c(u)$,  then if 
			\begin{align*}
t_c(u)= \left(\frac{1}{b_c(u)}\right)^{\frac{1}{2(2^*_{\mu}-1)}}
			\end{align*}
 we obtain $k^{\prime}(t_c(u))=0$ and $k^{\prime\prime}(t_c(u))<0$. Therefore, there exists a unique $t_c(u)>0$ such that 
		\begin{align*}
\max_{t\geq 0} \mathcal{I}_c(tu)= \mathcal{I}_c(t_c(u)u)= \frac{N-\mu+2}{2(2N-\mu)} \left(b_c(u)\right)^{-\frac{N-2}{N-\mu+2}}.
		\end{align*}
		\item [(ii)] For  $0<\omega <1$,  we have
		\begin{align*}
		\ds \bigg |\int_{\Om}fu~dx\bigg | \leq \|f\|_{H^{-1}}\|u\| \leq \frac{\omega}{2}\|u\|^2+ 
		\frac{1}{2\omega}\|f\|_{H^{-1}}^2,
		\end{align*}
	and for $u \in H_0^1(\Om)$ with $u^+\not \equiv 0$ by the above inequality  we get 
	\begin{align*}
	&  \frac{1-\omega}{2}\|u\|^2-\frac{1}{2.2^*_{\mu}}b(u) -\frac{1}{2\omega}\|f\|^2_{H^{-1}}
	 \leq \mathcal{J}_f(u) \leq \frac{1+\omega}{2}\|u\|^2-\frac{1}{2.2^*_{\mu}}b(u) +\frac{1}{2\omega}\|f\|^2_{H^{-1}}. 
	\end{align*}
		This implies that 
		\begin{align*}
		(1-\omega)\mathcal{I}_{\frac{1}{1-\omega}}(u)-\frac{1}{2\omega}\|f\|^2_{H^{-1}}  \leq \mathcal{J}_f(u)\leq 	(1+\omega)\mathcal{I}_{\frac{1}{1+\omega}}(u)+\frac{1}{2\omega}\|f\|^2_{H^{-1}}.
		\end{align*}
		
		\item [(iii)]  For each $u \in \Sigma$ and	$0<\omega<1$ using part (ii), we obtain 		
		\begin{align*}
		(1-\omega)\mathcal{I}_{\frac{1}{1-\omega}}(t_{\frac{1}{1-\omega}}(u) u)-\frac{1}{2\omega}\|f\|^2_{H^{-1}}  \leq \mathcal{J}_f(t^-(u)u)\leq 	(1+\omega)\mathcal{I}_{\frac{1}{1+\omega}}(t_{\frac{1}{1+\omega}}(u)u)+\frac{1}{2\omega}\|f\|^2_{H^{-1}}, 
		\end{align*}
		 this on applying part (i) we  get 
		\begin{align*}
		\mathcal{I}_{\frac{1}{1-\omega}}(t_{\frac{1}{1-\omega}}(u) u) 
		& = \frac{N-\mu+2}{2(2N-\mu)} b_{\frac{1}{1-\omega}}(u)^{-\frac{N-2}{N-\mu+2}}\\
		& = (1-\omega)^{\frac{N-2}{N-\mu+2}}\frac{N-\mu+2}{2(2N-\mu)} b(u)^{-\frac{N-2}{N-\mu+2}}= (1-\omega)^{\frac{N-2}{N-\mu+2}}\mathcal{J}(t^*u).
		\end{align*} 
		Therefore, we get 
		\begin{align*}
		(1-\omega)^{\frac{2N-\mu}{N-\mu+2}}\mathcal{J}(t^*u)- \frac{1}{2\omega}\|f\|_{H^{-1}}^2 \leq \mathcal{J}_f(t^- u)\leq (1+\omega)^{\frac{2N-\mu}{N-\mu+2}}\mathcal{J}(t^*u)- \frac{1}{2\omega}\|f\|_{H^{-1}}^2.
		\end{align*}
	\item [(iv)] Combining part (iii) with the fact that $\Upsilon_0= \frac{N-\mu+2}{2(2N-\mu)}S_{H,L}^{\frac{2N-\mu}{N-\mu+2}}>0$  contributes that 
		\begin{align*}
		\Upsilon_f^- (\Om)& > (1-\omega)^{\frac{2N-\mu}{N-\mu+2}}\Upsilon_0-\frac{1}{2\omega}\|f\|_{H^{-1}}^2 \\
		&  = (1-\omega)^{\frac{2N-\mu}{N-\mu+2}}\frac{N-\mu+2}{2(2N-\mu)}S_{H,L}^{\frac{2N-\mu}{N-\mu+2}}-\frac{1}{2\omega}\|f\|_{H^{-1}}^2.  
		\end{align*}
	Thus, there exists $e_{11}>0$ such that $\Upsilon_f^- (\Om) >0$ whenever $\|f\|_{H^{-1}}< e_{11}$
	\QED		
	\end{itemize}	
	\end{proof}	

\begin{Lemma}\label{nhlem16}
	If $\Om$ satisfies condition $(A)$ then there exists a $\de_0 >0$ such that : if $u \in \mathcal{N}$ with $\mathcal{J}(u)\leq 	\frac{N-\mu+2}{2(2N-\mu)}S_{H,L}^{\frac{2N-\mu}{N-\mu+2}} +\de_0 $, then $
	\ds \int_{\mathbb{R}^N} \frac{x}{|x|} |\na u |^2 ~dx  \not = 0$ 
\end{Lemma}
	\begin{proof}
		Let  $\{u_n\}\in \mathcal{N}$ such that $\mathcal{J}(u_n)= 	\frac{N-\mu+2}{2(2N-\mu)}S_{H,L}^{\frac{2N-\mu}{N-\mu+2}}+o(1)$ and $\ds \int_{\mathbb{R}^N} \frac{x}{|x|} |\na u_n |^2 ~dx  = 0$. Since $\{u_n\} \in \mathcal{N}$ therefore by Lemma \ref{nhlem14},  $\{u_n\}$ is a  Palais-Smale sequence of $\mathcal{J} $ at level $\frac{N-\mu+2}{2(2N-\mu)}S_{H,L}^{\frac{2N-\mu}{N-\mu+2}}$. Now using Theorem 4.4 of \cite{choqcoron} and Remark \ref{nhrem1}, we have 
		\begin{align*}
\|u_n- (\la_n^1)^{\frac{2-N}{2}}v_1((.-y_n^1)/\la_n^1)\|_{D^{1,2}(\mathbb{R}^N)} \ra 0,
		\end{align*}
where $v_1$ is a minimizer of $S_{H,L}$, $\la_n^1 \in \mathbb{R}^+,\; y_n^1 \in \overline{\Om}$. Moreover, if $n\ra \infty$ then  $\la_n^1 \ra 0 $,  $ \frac{y_n^1}{|y_n^1|}\ra y_0$  is the unit vector in $\mathbb{R}^N$. 
		Thus   we obtain
		\begin{align*}
		0= \ds \int_{\mathbb{R}^N} \frac{x}{|x|} |\na u_n |^2 ~dx  =& \ds \int_{\mathbb{R}^N} \frac{x}{|x|} \left(|\na u_n |^2- |\na (\la_n^1)^{\frac{2-N}{2}}v_1((.-y_n^1)/\la_n^1) |^2 \right)~dx\\
		&+  \ds \int_{\mathbb{R}^N} \frac{x}{|x|} |\na (\la_n^1)^{\frac{2-N}{2}}v_1((.-y_n^1)/\la_n^1) |^2 ~dx\\
		 = &o_n(1)+ \ds \int_{\mathbb{R}^N} \frac{y_n^1+\la_n^1 z}{|y_n^1+\la_n^1 z|} |\na v_1(z) |^2 ~dz\\
		 = &o_n(1)+ y_0 S_{H,L}^{\frac{2N-\mu}{N-\mu+2}},
		\end{align*}
as	$n\ra \infty$, which is not possible.\QED		
		\end{proof}

For $0<\e\leq \e_0$ (defined in Proposition \ref{nhprop12}), define $H_{\e}: \mathbb{S}^{N-1}\ra H_0^1(\Om)$ as
	\begin{align*}
H_{\e}(\sigma)=u_1+ s_0t_0 g_\rho^{\e,\sigma}, 
	\end{align*} 
	where   the function $u_1+ s_0t_0 g_\rho^{\e,\sigma}$ defined in Lemma \ref{nhlem13}.
	\begin{Lemma}
There exists a $\de_{\e}\in \mathbb{R}^+$ such that 
\begin{align*}
H_{\e}(\mathbb{S}^{N-1}) \subset \left[\mathcal{J}_f \leq \Upsilon_f(\Om)+	\frac{N-\mu+2}{2(2N-\mu)}S_{H,L}^{\frac{2N-\mu}{N-\mu+2}}-\de_{\e}\right].
\end{align*}
	\end{Lemma}
	\begin{proof}
	Trivially, $H_{\e}(\sigma)=u_1+ s_0t_0g_\rho^{\e,\sigma}\in \mathcal{N}_f^- $. So we only have to prove that $\mathcal{J}_f(u_1+ s_0t_0g_\rho^{\e,\sigma}) \leq \Upsilon_f(\Om)+	\frac{N-\mu+2}{2(2N-\mu)}S_{H,L}^{\frac{2N-\mu}{N-\mu+2}}-\de_{\e}$ for some $\de_{\e}>0$. Since by Proposition \ref{nhprop12},
		\begin{align*}
		\sup_{t\geq 0}\mathcal{J}_f(u_1+t g_\rho^{\e,\sigma})< \mathcal{J}_f(u_1)+	\frac{N-\mu+2}{2(2N-\mu)}S_{H,L}^{\frac{2N-\mu}{N-\mu+2}}  = \Upsilon_f(\Om)+	\frac{N-\mu+2}{2(2N-\mu)}S_{H,L}^{\frac{2N-\mu}{N-\mu+2}}.
		\end{align*}
		Hence there exists a $\de_{\e}>0$ such that 
		\[
		\mathcal{J}_f(u_1+ s_0t_0g_\rho^{\e,\sigma})\leq \sup_{t\geq 0}\mathcal{J}_f(u_1+t g_\rho^{\e,\sigma})\leq   \Upsilon_f(\Om)+	\frac{N-\mu+2}{2(2N-\mu)}S_{H,L}^{\frac{2N-\mu}{N-\mu+2}}-\de_{\e}.\]
		\QED
		\end{proof}
		\begin{Lemma}
			There exists a $e_{22}>0$ such that $\|f\|_{H^{-1}}< e_{22}$ then for  any 
			
			\[u \in \left[\mathcal{J}_f \leq \Upsilon_f(\Om)+	\frac{N-\mu+2}{2(2N-\mu)}S_{H,L}^{\frac{2N-\mu}{N-\mu+2}}\right] \text{ we have } \ds \int_{\mathbb{R}^N} \frac{x}{|x|} |\na u |^2 ~dx  \not = 0.\]
		\end{Lemma}
		\begin{proof}
			Let  $u \in \left[\mathcal{J}_f \leq \Upsilon_f(\Om)+	\frac{N-\mu+2}{2(2N-\mu)}S_{H,L}^{\frac{2N-\mu}{N-\mu+2}}\right] $  then $\mathcal{J}_f(u) \leq \Upsilon_f(\Om)+	\frac{N-\mu+2}{2(2N-\mu)}S_{H,L}^{\frac{2N-\mu}{N-\mu+2}}$ and $u \in \mathcal{N}_f^- $, that is, $\frac{1}{\|u\|} t^-\left(\frac{u}{\|u\|}\right)=1$. Since $ \Upsilon_f(\Om)<0$ we have 
			$\mathcal{J}_f(u) \leq 	\frac{N-\mu+2}{2(2N-\mu)}S_{H,L}^{\frac{2N-\mu}{N-\mu+2}}$. So for $\frac{u}{\|u\|}\in \Sigma$ there exits a $t^*>0$ such that $\frac{t^*u}{\|u\|}\in \mathcal{N}$ which on using Lemma \ref{nhlem15} (iii) implies 
			\begin{align*}
			(1-\omega)^{\frac{2N-\mu}{N-\mu+2}}\mathcal{J}\left(t^*\frac{u}{\|u\|}\right)- \frac{1}{2\omega}\|f\|_{H^{-1}}^2 \leq \mathcal{J}_f\left(t^- \frac{u}{\|u\|}\right)= \mathcal{J}_f(u). 
			\end{align*}
			Now using Lemma \ref{nhlem28}, we have 
			\begin{align*}
			\mathcal{J}\left(t^*\frac{u}{\|u\|}\right)  & \leq (1-\omega)^{- \frac{2N-\mu}{N-\mu+2}} \left(\mathcal{J}_f(u) + \frac{1}{2\omega}\|f\|_{H^{-1}}^2 \right)
	\\&  \leq (1-\omega)^{- \frac{2N-\mu}{N-\mu+2}} \left(	\frac{N-\mu+2}{2(2N-\mu)}S_{H,L}^{\frac{2N-\mu}{N-\mu+2}} + \frac{1}{2\omega}\|f\|_{H^{-1}}^2 \right) \\
	&=  \left( (1-\omega)^{- \frac{2N-\mu}{N-\mu+2}}-1\right)	\frac{N-\mu+2}{2(2N-\mu)}S_{H,L}^{\frac{2N-\mu}{N-\mu+2}} \\
	& \quad +  \left(	\frac{N-\mu+2}{2(2N-\mu)}S_{H,L}^{\frac{2N-\mu}{N-\mu+2}} + \frac{1}{2\omega  (1-\omega)^{ \frac{2N-\mu}{N-\mu+2}}}\|f\|_{H^{-1}}^2 \right).
			\end{align*}
			Choose $\omega_0>0$ such that  for $0<\omega<\omega_0$, we have 
			$\left( (1-\omega)^{- \frac{2N-\mu}{N-\mu+2}}-1\right)	\frac{N-\mu+2}{2(2N-\mu)}S_{H,L}^{\frac{2N-\mu}{N-\mu+2}} < \frac{\de_0}{2}$  where $\de_0$ is defined in Lemma \ref{nhlem16}.  Now for $0<\omega<\omega_0$ choose $ e_{22}$ such that if $\|f\|_{H^{-1}} < e_{22}$ then $ \ds \frac{1}{2\omega  (1-\omega)^{ \frac{2N-\mu}{N-\mu+2}}}\|f\|_{H^{-1}}^2 < \frac{\de_0}{2}$. Therefore, we obtain 
			\begin{align*}
			\mathcal{J}\left(t^*\frac{u}{\|u\|}\right)  & \leq 	\frac{N-\mu+2}{2(2N-\mu)}S_{H,L}^{\frac{2N-\mu}{N-\mu+2}} + \de_0
\end{align*}			
	Using 	 Lemma \ref{nhlem16} we conclude the result. \QED
			\end{proof}		
\noi  Define $G: [\mathcal{J}_f \leq \Upsilon_f(\Om)+	\frac{N-\mu+2}{2(2N-\mu)}S_{H,L}^{\frac{2N-\mu}{N-\mu+2}}]\ra \mathbb{S}^{N-1}$ by
	\[G(u )=\frac{\ds \int_{\mathbb{R}^N} \frac{x}{|x|} |\na u |^2 ~dx }{\bigg|\ds \int_{\mathbb{R}^N} \frac{x}{|x|} |\na u |^2 ~dx \bigg|}.\]
Note that from Lemma \ref{nhlem16}, $G$ is well defined.
\begin{Lemma}\label{nhlem17}
	For $0<\e< \e_0$ and $\|f\|_{H^{-1}}< e_{22}$, the map 
	\begin{align*}
	G \; o\; H_{\e}: \mathbb{S}^{N-1} \ra  \mathbb{S}^{N-1}
	\end{align*}
	is homotopic to the identity.
\end{Lemma}
		\begin{proof}
			Define  $\mathcal{K}:= \bigg\{u \in H_0^1(\Om)\setminus \{0\} \bigg| \ds \int_{\mathbb{R}^N} \frac{x}{|x|} |\na u |^2 ~dx  \not = 0 \bigg\}$ and   $\overline{G}: \mathcal{K} \ra \mathbb{S}^{N-1}$ by
			\begin{align*}
			\overline{G}(u )=\ds \int_{\mathbb{R}^N} \frac{x}{|x|} |\na u |^2 ~dx \bigg/ \bigg|\ds \int_{\mathbb{R}^N} \frac{x}{|x|} |\na u |^2 ~dx \bigg|		
				\end{align*}
			 as an extension of $G$. Lemma \ref{nhlem8}  and Lemma \ref{nhlem16},  gives $ \ds \int_{\mathbb{R}^N} \frac{x}{|x|} |\na g_\rho^{\e,\sigma} |^2 ~dx  \not = 0$ for sufficiently small $\e$. Thus, $\overline{G}(g_\rho^{\e,\sigma})$ is well defined. Now let $\gamma:[s_1,s_2]\ra \mathbb{S}^{N-1}$ be a regular geodesic between $\overline{G}(g_\rho^{\e,\sigma})$ and $\overline{G}(H_{\e}(\sigma))$	such that $\gamma(s_1)= \overline{G}(g_\rho^{\e,\sigma})$ and $\gamma(s_2)=\overline{G}(H_{\e}(\sigma))$. Moreover, by a analogous argument as in Lemma \ref{nhlem8}, for  $\de_0>0$ there exists a $\e_0>0$ such that 
\begin{align*}
\mathcal{J}(g_\rho^{2(1-\la)\e})< \frac{N-\mu+2}{2(2N-\mu)}S_{H,L}^{\frac{2N-\mu}{N-\mu+2}} +\de_0 \text{ for all } 0<\e<\e_0 \text{ and }\sigma \in \mathbb{S}^{N-1}, \la \in [\frac{1}{2},1),
 \end{align*}
 where $\de_0$ is defined in Lemma \ref{nhlem16}. Now define  $ \varsigma_{\e}(\la,\sigma): [0,1]\times \mathbb{S}^{N-1} \ra \mathbb{S}^{N-1}$ by
			\[ 
			\varsigma_{\e}(\la,\sigma)=  \left\{
			\begin{array}{ll}
				\gamma(2\la(s_1-s_2)+s_2)  &  \text{ if }\la \in [0,\frac{1}{2}), \\
		\overline{G}(g_\rho^{2(1-\la)\e}) & \text{ if } \la \in [\frac{1}{2},1), \\
			\sigma  & \text{ if } \la =1 .\\
			\end{array} 
			\right. \]	
		 	Clearly, $\varsigma_{\e}$ is well defined. We claim that $\ds \lim_{\la\ra 1^-}\varsigma_{\e}(\la,\sigma)= \sigma $ and $\ds \lim_{\la\ra \frac{1}{2}^-}\varsigma_{\e}(\la,\sigma)=\overline{G}(g_\rho^{\e,\sigma})$.
\begin{enumerate}		 	
		 	\item[ (i)] $\ds \lim_{\la\ra 1^-}\varsigma_{\e}(\la,\sigma)= \sigma $: Indeed
\begin{align*}
\ds \int_{\mathbb{R}^N} \frac{x}{|x|} |\na g_\rho^{2(1-\la )\e} |^2 ~dx = S_{H,L}^{\frac{2N-\mu}{N-\mu+2}}\sigma +o(1) \text{ as } \la\ra 1^- 
\end{align*}
then  $\ds \lim_{\la\ra 1^-}\varsigma_{\e}(\la,\sigma)= \sigma $. \item[(b)] $\ds \lim_{\la\ra \frac{1}{2}^-}\varsigma_{\e}(\la,\sigma)=\overline{G}(g_\rho^{\e,\sigma})$: Indeed
\begin{align*}
\ds \lim_{\la\ra \frac{1}{2}^-}\varsigma_{\e}(\la,\sigma)=\ds \lim_{\la\ra \frac{1}{2}^-} \gamma(2\la(s_1-s_2)+s_2)= \gamma(s_1)= \overline{G}(g_\rho^{\e,\sigma}).
\end{align*}\end{enumerate}
Hence, $\varsigma_{\e} \in C([0,1]\times \mathbb{S}^{N-1} , \mathbb{S}^{N-1})$ and $\varsigma_{\e}(0,\sigma)= \overline{G}(H_{\e}(\sigma))$  and $\varsigma_{\e}(1,\sigma)=\sigma$ for $\sigma \in \mathbb{S}^{N-1}$ provided $0<\e<\e_0$  and $\|f\|_{H^{-1}}<e_{22}$. Thus the result follows. \QED
			\end{proof}
		\begin{Proposition}\label{nhprop2} Let $e^*:=\min\{e_{00}, e_{11},e_{22}\}$ and let 
			 $\|f\|_{H^{-1}}<e^* $ then $\mathcal{J}_f$  has two critical points in 
			\begin{align*}
			\bigg [\mathcal{J}_f \leq \Upsilon_f(\Om)+	\frac{N-\mu+2}{2(2N-\mu)}S_{H,L}^{\frac{2N-\mu}{N-\mu+2}}\bigg ] .
			\end{align*}
			Equivalently, $(P_f)$ have another two different solutions which are different from $u_1$.
		\end{Proposition}
		\begin{proof} 
			Using Lemma \ref{nhlem17} and Lemma \ref{nhlem32}, we have
			\begin{align*}
cat\bigg (\bigg [\mathcal{J}_f \leq \Upsilon_f(\Om)+	\frac{N-\mu+2}{2(2N-\mu)}S_{H,L}^{\frac{2N-\mu}{N-\mu+2}}- \de_{\e} \bigg ] \bigg)\geq 2.
			\end{align*} 
			Now the proof follows from  Lemma \ref{nhlem18}(i) and Lemma \ref{nhlem31}.\QED  
			\end{proof}
\section{Existence of Fourth solution}
In this section we will prove the existence of high energy solution by using Brouwer's degree theory and minmax theorem given by Brezis and Nirenberg \cite{brezis2}.

  Let $\mathcal{V}:= \bigg\{ u \in H_0^1(\Om): \ds \int_{\Om}\int_{\Om} \frac{|u^+(x)|^{2^{*}_{\mu}}|u^+(y)|^{2^{*}_{\mu}}}{|x-y|^{\mu}}~dxdy=1 \bigg\},\; h_\rho^{\e,\sigma}(x)= \ds \frac{g_\rho^{\e,\sigma}(x)}{\|g_\rho^{\e,\sigma}\|_{NL}}$ where $g_\rho^{\e,\sigma}$ is defined in \eqref{nh8} 
 
\begin{Lemma}\label{nhlem20}
	$\|h_\rho^{\e,\sigma}\|^2_{D^{1,2}(\mathbb{R}^N)} \ra S_{H,L}$ as $\e \ra 0$ uniformly in $\sigma \in \mathbb{S}^{N-1}$. 
\end{Lemma}
\begin{proof}
Proof follows from Lemma \ref{nhlem8}\textit{(i)}. \QED
	\end{proof}

\begin{Lemma}\label{nhlem22}
	There exists a $\rho_0>0$ such that for $0<\rho<\rho_0$,
	\begin{align*}
	\ds \sup _{\sigma \in \mathbb{S}^{N-1}, \e\in (0,1]}\| h_\rho^{\e,\sigma}\|^2 < 2^{\frac{N-\mu+2}{2N-\mu}}S_{H,L}.
	\end{align*} 
\end{Lemma}
\begin{proof} Since we know that $\|\na u_{\e}^{\sigma}\|^2_{L^2(\mathbb{R}^N)}=\|u_{\e}^{\sigma}\|_{NL}^{2.2^*_{\mu}}= S_{H,L}^{\frac{2N-\mu}{N-\mu+2}} $ and  this on using  Lemma \ref{nhlem19}  we get   $\ds \sup _{\sigma \in \mathbb{S}^{N-1}, \e\in (0,1]}\| h_\rho^{\e,\sigma}\|^2\ra  S_{H,L}$ as $\rho\ra 0$. So there exists a $\rho_0$ such that  $0<\rho<\rho_0$, we obtain $\ds \sup _{\sigma \in \mathbb{S}^{N-1}, \e\in (0,1]}\| h_\rho^{\e,\sigma}\|^2 < 2^{\frac{N-\mu+2}{2N-\mu}}S_{H,L}$. \QED
	\end{proof}
Now for any $u\in H^1_0(\Om)$, by extending it to be zero outside $\Om$,  we define Barycenter mapping  $\ba: \mathcal{V} \ra  \mathbb{R}^N$ as  \[\ba(u)= \ds \int_{\mathbb{R}^N}\int_{\mathbb{R}^N} \frac{x|u^+(x)|^{2^{*}_{\mu}}|u^+(y)|^{2^{*}_{\mu}}}{|x-y|^{\mu}}~dxdy,\]
   and also let \begin{align*}
\mathcal{Q}:=\{u \in \mathcal{V}: \ba(u)=0\}.
\end{align*}

\begin{Lemma}
	There holds $\ds \lim_{\e \ra 0} \ba(h_\rho^{\e,\sigma})= \sigma$.
\end{Lemma}
\begin{proof}
	If there exists  $\eta>0$ and a sequence $\e_n \ra 0^+$ such that $|\ba(h_\rho^{\e_n})- \sigma|\geq \eta$. 	Then 
	\begin{align*}
	&\ba(h_\rho^{\e_n})  = \frac{\ds \int_{\mathbb{R}^N}\int_{\mathbb{R}^N} \frac{x|h_\rho^{\e_n}(x)|^{2^{*}_{\mu}}|h_\rho^{\e_n}(y)|^{2^{*}_{\mu}}}{|x-y|^{\mu}}~dxdy}{\|h_\rho^{\e_n}\|_{NL}^{2.2^*_{\mu}}}\\
	& = \sigma+   \frac{ \e_n \ds \int_{\mathbb{R}^N}\int_{\mathbb{R}^N} \frac{(z-\sigma)|\upsilon_\rho(\e_nz+(1-\e_n)\sigma)|^{2^{*}_{\mu}}|\upsilon_\rho(\e_nw +(1-\e_n)\sigma)|^{2^{*}_{\mu}}}{|z-w|^{\mu} [1+|z|^2]^{\frac{2N-\mu}{2}} [1+|w|^2]^{\frac{2N-\mu}{2}} } ~dzdw} {\ds \int_{\mathbb{R}^N}\int_{\mathbb{R}^N} \frac{|\upsilon_\rho(\e_nz+(1-\e_n)\sigma)|^{2^{*}_{\mu}}|\upsilon_\rho(\e_nw +(1-\e_n)\sigma)|^{2^{*}_{\mu}}}{|z-w|^{\mu} [1+|z|^2]^{\frac{2N-\mu}{2}} [1+|w|^2]^{\frac{2N-\mu}{2}} } ~dzdw}\\
	& \leq \sigma+  \e_n \sup_{z \in supp(\upsilon_\rho)} |z-\sigma| \leq  \sigma +C \e_n, \text{ for some } C>0.
	\end{align*}
	It implies that $0< \eta \leq|\ba(h_\rho^{\e_n})- \sigma|\leq C \e_n \ra 0^+ $ as $\e_n \ra 0^+$, a contradiction.
	\QED	
\end{proof}

\begin{Lemma}\label{nhlem23}
	Let $m_0= \ds \inf_{u \in \mathcal{Q}}\|u\|^2$ then $S_{H,L}< m_0$. 
\end{Lemma}
\begin{proof}
	Obviously $S_{H,L} \leq m_0$,  so let if possible, $S_{H,L}= \inf_{u \in \mathcal{Q}}\|u\|^2$ then there exists a sequence $\{v_n\} \in H_0^1(\Om)$ such that $\|v_n\|_{NL}=1,\; \ba(v_n)=0,\; \|v_n\|^2 \ra S_{H,L}$ as $n \ra \infty$. Setting $w_n= S_{H,L}^{\frac{N-2}{2(N-\mu+2)}}v_n$ we get  $\|w_n\|_{NL}^{2.2^*_{\mu}} =  S_{H,L}^{\frac{2N-\mu}{N-\mu+2}}$ and $\|w_n\|^2\ra   S_{H,L}^{\frac{2N-\mu}{N-\mu+2}}$. Therefore, $\mathcal{J}(w_n)\ra \frac{N-\mu+2}{2(2N-\mu)}S_{H,L}^{\frac{2N-\mu}{N-\mu+2}} $ and $\mathcal{J}^{\prime}(w_n)(w_n)= o(1)$. Using Lemma \ref{nhlem14}, we obtain $\{w_n\}$ is a Palais-Smale sequence of $\mathcal{J}$ at level $	\frac{N-\mu+2}{2(2N-\mu)}S_{H,L}^{\frac{2N-\mu}{N-\mu+2}}$. Subsequently, by 
	Theorem 4.4 of \cite{choqcoron} and Remark \ref{nhrem1}, there exist sequences $y_n \in \Om, \; \la_n \in \mathbb{R}^+$ such that $y_n \ra y_0\in \overline{\Om}$ and $\la_n\ra 0$, for the functions 
	\begin{align*}
	v_n = S_{H,L}^{- \frac{N-2}{2(N-\mu+2)}} w_{n}, \text{ where }  w_{n}= C \left(\frac{\la_n}{\la_n^2+|x-y_n|^2}\right)^{\frac{N-2}{2}} \text{ for some } C>0.
	\end{align*} 
	Thus if   
	\begin{align*}
	& C_1=  C\int_{\mathbb{R}^N}\int_{\mathbb{R}^N} \frac{z}{|z-w|^{\mu} [1+|z|^2]^{\frac{2N-\mu}{2}} [1+|w|^2]^{\frac{2N-\mu}{2}} } ~dzdw \quad \text{ and }\\ & C_2=  C \int_{\mathbb{R}^N}\int_{\mathbb{R}^N} \frac{1}{|z-w|^{\mu} [1+|z|^2]^{\frac{2N-\mu}{2}} [1+|w|^2]^{\frac{2N-\mu}{2}} } ~dzdw, 
	\end{align*}
then 	 \begin{align*}
	 0= \ba(v_n)= C\int_{\mathbb{R}^N}\int_{\mathbb{R}^N} \frac{x|v_n(x)|^{2^{*}_{\mu}}|v_n(y)|^{2^{*}_{\mu}}}{|x-y|^{\mu}}~dxdy = \la_n C_1 +y_n C_2 \ra C_2y_0. 
	 \end{align*}
	 This is a contradiction. Hence $S_{H,L}< m_0$. \QED
	\end{proof}
\begin{Lemma}\label{nhlem21}
	There exists $\e_0>0$ such that for $0<\e<\e_0$ and $|\sigma|=1$ we have 
	\begin{align*}
	S_{H,L}< \|h_\rho^{\e,\sigma}\|^2_{D^{1,2}(\mathbb{R}^N)} < \frac{m_0+ S_{H,L}}{2}.
	\end{align*}
\end{Lemma}
\begin{proof} Apparently $	S_{H,L}\leq  \|h_\rho^{\e,\sigma}\|^2_{D^{1,2}(\mathbb{R}^N)}$ and we know that $S_{H,L}$ is not attained on a bounded domain. Thus, $S_{H,L}< \|h_\rho^{\e,\sigma}\|^2_{D^{1,2}(\mathbb{R}^N)}$. Since $S_{H,L}< m_0$,  there exists $\de_0$ such that $\frac{S_{H,L}}{2}+ \de_0<\frac{m_0}{2}$ and from Lemma \ref{nhlem20} we know that $\|h_\rho^{\e,\sigma}\|^2_{D^{1,2}(\mathbb{R}^N)}\ra S_{H,L}$ as $\e \ra 0$. Therefore for $\de_0>0$ there exists a $\e_0>0$ such that  $\|h_\rho^{\e,\sigma}\|^2_{D^{1,2}(\mathbb{R}^N)}< S_{H,L}+\de_0 $ whenever $0<\e<\e_0$. Hence we have the desired result. \QED
	\end{proof}
 Now we will state the minimax lemma given by Brezis and Nirenberg \cite{brezis2}. 
\begin{Lemma}\label{nhlem25}
	Let $Y$ be a  Banach space and $\phi \in C^1(Y, \mathbb{R})$. Let $A$ be a compact metric space, $A_0 \subset A$ be a closed set and $\gamma \in C(A_0, \; Y)$. Define
	\begin{align*}
	\Gamma= \{g \in C(A, \; Y): g(s)= \gamma(s) \text{ if } s \in A_0\},\; \;  \overline{c}= \ds \inf_{g \in \Gamma}\sup_{s\in A} \phi(g(s)),\; \; \hat{c}= \ds \sup_{\gamma(A_0)} \phi.
	\end{align*}
 \noi If  $\overline{c}> \hat{c}$ then there exists a sequence $\{u_n\} \in Y$ satisfying 
	$
	\phi(u_n) \ra  \overline{c} \text{ and }\phi^{\prime}(u_n)\ra 0.
$. Further, if  $\phi$ satisfies $(PS)_{\overline{c}}$ condition then  there exists $u_0\in Y$ such that $\phi(u_0)= \overline{c}$ and $\phi^\prime(u_0)=0$.\QED
\end{Lemma}
Let $r_0= 1-\e_0$ and $\overline{B}_{r_0}= \{(1-\e)\sigma \in \mathbb{R}^N: |(1-\e)\sigma|\leq r_0,\;  \sigma \in \mathbb{S}^{N-1},\;  0< \e\leq 1 \}$, where $\e_0$ is defined in Lemma \ref{nhlem21}. Then we set 
$\ds F= \{q \in C(\overline{B}_{r_0}, \mathcal{V}); \; q_{| \pa \overline{B}_{r_0}} = h_\rho^{\e,\sigma} \}$
and 
\begin{align*}
 \overline{c}= \ds \inf_{q \in F}\sup_{(1-\e)\sigma  \in \overline{B}_{r_0}}\|q((1-\e)\sigma)\|^2, \qquad \hat{c}= \ds \sup_{\pa \overline{B}_{r_0}} \|h_\rho^{\e,\sigma}\|^2
\end{align*}
\begin{Lemma}
For each $q \in F$, we have $q(\overline{B}_{r_0})\cap \mathcal{Q} \not = \emptyset$. 
\end{Lemma}
\begin{proof}
	It is enough to show there exist $\tilde{e}>0$ and $\tilde{\sigma}\in \mathbb{S}^{N-1}$ such that $\ba(q((1-\tilde{e})\tilde{\sigma}))=0$. Define $\psi:\overline{B}_{r_0}\ra  \mathbb{R}^N $ by $\psi((1-\e)\sigma)= \ba (q((1-\e)\sigma))$.  We claim that 
	\begin{align*}
d(\psi,\overline{B}_{r_0} , 0 )= d(I,\overline{B}_{r_0} , 0 )\not = 0, \text{ where } d \text{  is Brouwer's  topological degree}.
	\end{align*}
If $(1-\e)\sigma \in \pa \overline{B}_{r_0} $ then $q((1-\e)\sigma)= h_\rho^{\e,\sigma}$ which implies
\begin{align*}
\psi((1-\e)\sigma)= \ba (q((1-\e)\sigma))= \ba(h_\rho^{\e,\sigma})=\sigma+o(1) \text{ as } \e \ra 0 .
\end{align*}  
Now define the homotopy $\mathcal{H}:[0,1]\times 	\overline{B}_{r_0} \ra  \mathbb{R}^N $ by \begin{align*}
\mathcal{H}(t, (1-\e)\sigma)= (1-t)\psi((1-\e)\sigma)+ tI((1-\e)\sigma)
\end{align*}
then for $(1-\e)\sigma \in \pa \overline{B}_{r_0}$ and $t\in[0,1] $ we have
\begin{align*} 
\mathcal{H}(t, (1-\e)\sigma) &= (1-t)\sigma+o(1)+ t(1-\e_0)\sigma \\
& = o(1)+(1-\e_0t)\sigma \not = 0, \qquad  \text{ as } \e \ra 0. 
\end{align*}
So by Brouwer's degree theory, claim holds. It implies there exist $\tilde{e}>0$ and $\tilde{\sigma}\in \mathbb{S}^{N-1}$ such that $\psi((1-\tilde{e})\tilde{\sigma})=0$ that is,  $\ba(q((1-\tilde{e})\tilde{\sigma}))=0$. \QED	
	\end{proof}

\noi Using above Lemma we have $m_0\leq 
\ds\sup_{(1-\e)\sigma  \in \overline{B}_{r_0}}\|q((1-\e)\sigma)\|^2 $ for all $q \in F$. Hence \begin{align*}
m_0\leq \ds \inf_{q \in F} \sup_{(1-\e)\sigma  \in \overline{B}_{r_0}}\|q((1-\e)\sigma)\|^2  = \overline{c}. 
\end{align*} 
Also, by the definition of $\overline{c}$, and Lemma \ref{nhlem22},  we have $\overline{c} < 2^{\frac{N-\mu+2}{2N-\mu}}S_{H,L}$ for $0<\rho<\rho_0$. Combining all these and using Lemma \ref{nhlem23} we have 
\begin{align}\label{nh9}
S_{H,L}< m_0\leq \overline{c}< 2^{\frac{N-\mu+2}{2N-\mu}}S_{H,L} \text{ for } \rho \text{ sufficiently small }.
\end{align}
In addition, from Lemma  \ref{nhlem21},  we get  \[\hat{c}=  \ds \sup_{\pa \overline{B}_{r_0}} \|h_\rho^{\e,\sigma}\|^2 < \frac{m_0+ S_{H,L}}{2} < m_0\leq \overline{c}. \]
Now we define 
\begin{align*}
& \breve{\mathcal{J}}_f(u)= \max_{t>0} \mathcal{J}_f(tu):\mathcal{V}\ra \mathbb{R}^N \quad \text{and } \quad \breve{\mathcal{J}}(u)= \max_{t>0}\mathcal{J}(tu):\mathcal{V}\ra \mathbb{R}^N ,\\
& \gamma_f= \ds \inf_{q \in F}\sup_{(1-\e)\sigma  \in \overline{B}_{r_0}} \breve{\mathcal{J}}_f(q((1-\e)\sigma))\quad \text{and } \quad  \gamma_0= \ds \inf_{q \in F}\sup_{(1-\e)\sigma  \in \overline{B}_{r_0}} \breve{\mathcal{J}}(q((1-\e)\sigma)) . 
\end{align*}
We remark that the conclusion of Lemma \ref{nhlem15} (iii) holds true for $\breve{\mathcal{J}}_f$. Moreover, $\breve{\mathcal{J}}_f(u)=\ds  \max_{t>0} \mathcal{J}_f(tu)=\mathcal{J}_f( t^-(u)u) $, where $t^-(u)$ is defined in Lemma \ref{nhlem3}.

 \begin{Lemma}\label{nhlem35}
	The following holds:
	\begin{enumerate}
		\item [(i)] $\breve{\mathcal{J}}_f \in C^1(\mathcal{V},\mathbb{R})$ and $\ld \breve{\mathcal{J}}_f^{\prime}(u),h\rd = t^-(u)\ld \mathcal{J}_f^{\prime}( t^-(u)u),h\rd$ for all $ h \in T_u(\mathcal{V})$ where \\ $T_u(\mathcal{V}):=\bigg\{ h \in H_0^1(\Om) \bigg| \ds \int_{\Om}\int_{\Om} \frac{|u^+(x)|^{2^*_{\mu}}|u^+(y)|^{2^*_{\mu}-1}h(y)}{|x-y|^{\mu}}=0 \bigg\}$.
		\item [(ii)] If $u \in \mathcal{V}$ is a critical point of $\breve{\mathcal{J}}_f$ then $t^-(u)u \in \mathcal{N}_f^-$ is a critical point of $\mathcal{J}_f$.
		\item[(iii)] If $\{u_n\}_{n\in \mathbb{N}}$  is a $(PS)_c$ sequence of $\breve{\mathcal{J}}_f$ then $\{t^-(u_n)u_n\}_{n\in \mathbb{N}} \in \mathcal{N}_f^-$ is a $(PS)_c$ sequence for $\mathcal{J}_f$.
		\end{enumerate}
		\end{Lemma}
\begin{proof}
	(a) As for every $u \in H_0^1(\Om),\; t^-(u)u \in \mathcal{N}_f^-$ that is, $\ld \mathcal{J}_f^{\prime}( t^-(u)u),u\rd=0$  and 
	$\frac{d^2}{dt^2}\bigg |_{t=t^-(u) } \mathcal{J}_f(tu)<0$ . Therefore, by implicit function theorem, we get $t^-(u) \in C^1(\mathcal{V}, (0,\infty))$. As a result, $\breve{\mathcal{J}}_f(u)=\mathcal{J}_f( t^-(u)u) \in C^1(\mathcal{V},\mathbb{R})$ and for all $ h \in T_u(\mathcal{V})$,  we have 
	\begin{align*}
	\ld \breve{\mathcal{J}}_f^{\prime}(u),h\rd = t^-(u)\ld \mathcal{J}_f^{\prime}( t^-(u)u),h\rd+ \ld \mathcal{J}_f^{\prime}( t^-(u)u),u\rd \ld( t^-(u))^{\prime},h \rd = t^-(u)\ld \mathcal{J}_f^{\prime}( t^-(u)u),h\rd.
	\end{align*}
	(b) Combining the fact that  $u \in \mathcal{V}$ is a critical point of $\breve{\mathcal{J}}_f$  and $\ld \mathcal{J}_f^{\prime}( t^-(u)u),u\rd=0$, we get the desired result.\\
	(c) Let $\{u_n\}_{n\in \mathbb{N}}$  is a $(PS)_c$ sequence of $\breve{\mathcal{J}}_f$, that is, $u_n \in \mathcal{V},$
	\begin{align*}
	\breve{\mathcal{J}}_f(u_n) \ra c \text{ and }  \|\breve{\mathcal{J}}_f^{\prime}(u)\|_{T^*_{u_n}(\mathcal{V})}=\sup \{ |\ld \breve{\mathcal{J}}_f^{\prime}(u_n),h\rd| : h \in T_{u_n}(\mathcal{V}), \|h\|=1 \} \ra 0 \text{ as  } n\ra \infty.
	\end{align*}
	By Lemma \ref{nhlem3} we have $t^-(u_n)> \left(\frac{\|u\|^2}{2.2^*_{\mu}-1}\right)^{\frac{1}{2.2^*_{\mu}-2}}> \left(\frac{S_{H,L}}{2.2^*_{\mu}-1}\right)^{\frac{1}{2.2^*_{\mu}-2}}>C$ for some $C>0$. Since $H_0^1(\Om)= R_{u_n}\oplus T_{u_n}(\mathcal{V})$ so $\ld \mathcal{J}_f^{\prime}(u_n),v \rd = \ld \mathcal{J}_f^{\prime}(u_n),h_v \rd$,  where $h_v$ is the projection of $v$ in $T_{u_n}(\mathcal{V})$. Hence,
	\begin{align*}
	\|\mathcal{J}_f^{\prime}(t^-(u_n)u_n)\| 
	& = \sup_{v \in H_0^1(\Om), \|v\|=1}|\ld \mathcal{J}_f^{\prime}(t^-(u_n)u_n),v \rd| \\
	& = \sup_{v \in H_0^1(\Om), \|v\|=1}|\ld \mathcal{J}_f^{\prime}(t^-(u_n)u_n),h_v \rd| \\
	& = \sup_{v \in H_0^1(\Om), \|v\|=1} \frac{1}{t^-(u_n)}|\ld \breve{\mathcal{J}}_f^{\prime}(u_n),h_v \rd| \leq \frac{1}{C} \|\breve{\mathcal{J}}_f^{\prime}(u)\|_{T^*_{u_n}(\mathcal{V})} \ra 0.
	\end{align*}
Clearly, $\mathcal{J}_f( t^-(u_n)u_n) \ra c$. Therefore, $\{t^-(u_n)u_n\}_{n\in \mathbb{N}} \in \mathcal{N}_f^-$ is a $(PS)_c$ sequence for $\mathcal{J}_f$.
\QED
	\end{proof}

\begin{Lemma}\label{nhlem24}
	If $0<\rho<\rho_0$, then
	$\ds	\frac{N-\mu+2}{2(2N-\mu)}S_{H,L}^{\frac{2N-\mu}{N-\mu+2}} < \gamma_0< 	\frac{N-\mu+2}{2N-\mu}S_{H,L}^{\frac{2N-\mu}{N-\mu+2}}.$ 
\end{Lemma}
\begin{proof}
For $ u \in \mathcal{V}$, solving $\mathcal{J}^{\prime}(tu)= t\; a(u)- t^{2.2^*_{\mu}-1}=0$ we get $t=0$ and $t= (a(u))^{\frac{1}{2.2^*_{\mu}-2}}$. Therefore, 
	\begin{align*}
	\breve{\mathcal{J}}(u)= \max_{t>0} \mathcal{J}(tu)= \frac{N-\mu+2}{2(2N-\mu)}  \|u\|^{\frac{2(2N-\mu)}{N-\mu+2}}. 
	\end{align*}
	From the definition of $\overline{c}$, we obtain
	\begin{align*}
\gamma_0= \ds  \frac{N-\mu+2}{2(2N-\mu)}  \inf_{q \in F}\sup_{(1-\e)\sigma  \in \overline{B}_{r_0}}  \|q((1-\e)\sigma)\|^{\frac{2(2N-\mu)}{N-\mu+2}}=\frac{N-\mu+2}{2(2N-\mu)} \overline{c}^{\frac{2N-\mu}{N-\mu+2}}
	\end{align*} 
which on using \eqref{nh9} yields the desired result. \QED
	\end{proof}

\begin{Lemma}\label{nhlem26}
	$\breve{\mathcal{J}}_f(h_\rho^{\e,\sigma})= \frac{N-\mu+2}{2(2N-\mu)} S_{H,L}^{\frac{2N-\mu}{N-\mu+2}}+o(1)$ as $\e\ra 0$.
\end{Lemma}
\begin{proof}
	Since by Lemma \ref{nhlem8} it holds that  $h_\rho^{\e,\sigma}\rp 0$ in $H_0^1(\Om)$  as $\e\ra 0$.  On solving 
	\[
	\mathcal{J}_f^{\prime}(th_\rho^{\e,\sigma})
	= t ~a(h_\rho^{\e,\sigma})- t^{2.2^*_{\mu}-1}- \int_{\Om}fh_\rho^{\e,\sigma}~dx  =0,
	\]
	we conclude  $t_f= \|g_\rho^{\e,\sigma}\|_{NL}+o(1)$. Hence again from the Lemma \ref{nhlem8} we obtain\\
	$\ds
	\breve{\mathcal{J}}_f(h_\rho^{\e,\sigma})= \max_{t>0}\mathcal{J}_f(t h_\rho^{\e,\sigma})= \mathcal{J}_f(t_f h_\rho^{\e,\sigma})= \mathcal{J}_f( g_\rho^{\e,\sigma})= \frac{N-\mu+2}{2(2N-\mu)} S_{H,L}^{\frac{2N-\mu}{N-\mu+2}}+o(1) \text{ as } \e\ra 0.
	$
	\QED	
	\end{proof}
\begin{Lemma}\label{nhlem29}
	There exists $e_0^*>0$ such that if  $0< \|f\|_{H^{-1}}<e_0^*$, 
	\begin{align*}
	\Upsilon_f(\Om)+	\frac{N-\mu+2}{2(2N-\mu)}S_{H,L}^{\frac{2N-\mu}{N-\mu+2}}< \gamma_f< \Upsilon_f^-(\Om)+	\frac{N-\mu+2}{2(2N-\mu)}S_{H,L}^{\frac{2N-\mu}{N-\mu+2}}. 
	\end{align*}
\end{Lemma}
\begin{proof}
	Analogous to the proof of Lemma  \ref{nhlem15}(iii) we can have 
		\begin{align*}
	(1-\omega)^{\frac{2N-\mu}{N-\mu+2}}\mathcal{J}(t^*u)- \frac{1}{2\omega}\|f\|_{H^{-1}}^2 \leq \mathcal{J}_f(t^- u)\leq (1+\omega)^{\frac{2N-\mu}{N-\mu+2}}\mathcal{J}(t^*u)+ \frac{1}{2\omega}\|f\|_{H^{-1}}^2
	\end{align*}
	Using the above inequality  with the definition of $\breve{\mathcal{J}}$ and $\breve{\mathcal{J}}_f$, we get 
	\begin{align*}
	(1-\omega)^{\frac{2N-\mu}{N-\mu+2}}\breve{\mathcal{J}}(u)- \frac{1}{2\omega}\|f\|_{H^{-1}}^2 \leq \breve{\mathcal{J}}_f( u)\leq (1+\omega)^{\frac{2N-\mu}{N-\mu+2}}\breve{\mathcal{J}}(u)+ \frac{1}{2\omega}\|f\|_{H^{-1}}^2. 
	\end{align*}
	For $\de>0$ there exists $e_1(\de)$ such that  if  $\|f\|_{H^{-1}}<e_1(\de)$ then 
	\begin{align}\label{nh10}
	\gamma_0- \de< \gamma_f< \gamma_0+ \de.
	\end{align}
	\noi Now from Lemma \ref{nhlem15}(iii) for each  $0<\omega<1$, we have 
	\begin{align*}
	(1-\omega)& ^{\frac{2N-\mu}{N-\mu+2}}	\frac{N-\mu+2}{2(2N-\mu)}S_{H,L}^{\frac{2N-\mu}{N-\mu+2}}- \frac{1}{2\omega}\|f\|_{H^{-1}}^2\\&  \leq \Upsilon_f^-(\Om)\leq (1+\omega)^{\frac{2N-\mu}{N-\mu+2}}	\frac{N-\mu+2}{2(2N-\mu)}S_{H,L}^{\frac{2N-\mu}{N-\mu+2}}+ \frac{1}{2\omega}\|f\|_{H^{-1}}^2. 
	\end{align*}
	So for $\de>0$ there exists $e_2(\de)>0$ such that whenever $\|f\|_{H^{-1}}<e_2(\de)$ then 
	\begin{align*}
		\frac{N-\mu+2}{2(2N-\mu)}S_{H,L}^{\frac{2N-\mu}{N-\mu+2}}- \de \leq \Upsilon_f^-(\Om)\leq 	\frac{N-\mu+2}{2(2N-\mu)}S_{H,L}^{\frac{2N-\mu}{N-\mu+2}}+ \de . 
	\end{align*}
	It implies 
	\begin{align}\label{nh11}
	\frac{N-\mu+2}{2N-\mu}S_{H,L}^{\frac{2N-\mu}{N-\mu+2}} - \de \leq \Upsilon_f^-(\Om)+	\frac{N-\mu+2}{2(2N-\mu)}S_{H,L}^{\frac{2N-\mu}{N-\mu+2}}\leq \frac{N-\mu+2}{2N-\mu}S_{H,L}^{\frac{2N-\mu}{N-\mu+2}} + \de. 
	\end{align}
	Moreover, from Lemma \ref{nhlem24} \begin{align*}
		\frac{N-\mu+2}{2(2N-\mu)}S_{H,L}^{\frac{2N-\mu}{N-\mu+2}} < \gamma_0< 	\frac{N-\mu+2}{2N-\mu}S_{H,L}^{\frac{2N-\mu}{N-\mu+2}} . 
	\end{align*}
	Hence for fix small $0<\e<\min \bigg \{\frac{	\frac{N-\mu+2}{2N-\mu}S_{H,L}^{\frac{2N-\mu}{N-\mu+2}} -\gamma_0}{2}, \gamma_0- 	\frac{N-\mu+2}{2(2N-\mu)}S_{H,L}^{\frac{2N-\mu}{N-\mu+2}}\bigg\} $ such that if  $\|f\|_{H^{-1}}< e_0^*=  \min\{e_2(\e),\; e_2(\e) \}$ then using \eqref{nh10} and \eqref{nh11}, we obtain 
\begin{align*}
& 	\Upsilon_f(\Om)+	\frac{N-\mu+2}{2(2N-\mu)}S_{H,L}^{\frac{2N-\mu}{N-\mu+2}}< 	\frac{N-\mu+2}{2(2N-\mu)}S_{H,L}^{\frac{2N-\mu}{N-\mu+2}}< \gamma_0-\e \leq  \gamma_f \quad \text{ and }\\
& \gamma_f< \gamma_0+ 2\e-\e < 	\frac{N-\mu+2}{2N-\mu}S_{H,L}^{\frac{2N-\mu}{N-\mu+2}} -\e \leq \Upsilon_f^-(\Om)+	\frac{N-\mu+2}{2(2N-\mu)}S_{H,L}^{\frac{2N-\mu}{N-\mu+2}}.	
\end{align*}
That is,  $	\Upsilon_f(\Om)+	\frac{N-\mu+2}{2(2N-\mu)}S_{H,L}^{\frac{2N-\mu}{N-\mu+2}}<   \gamma_f < \Upsilon_f^-(\Om)+	\frac{N-\mu+2}{2(2N-\mu)}S_{H,L}^{\frac{2N-\mu}{N-\mu+2}}.$
\QED	
\end{proof}
\begin{Proposition} \label{nhprop3}
	If  $0<\rho<\rho_0$,  $0< \|f\|_{H^{-1}}<e_0^*$ (defined in Lemma \ref{nhlem29}) then there exists a critical point $u_4 \in \mathcal{N}_f^-$ of $\mathcal{J}_f$ with $\mathcal{J}_f(u_4)=\gamma_f$. 
\end{Proposition}
\begin{proof}
	  Let  $c \in \left(\Upsilon_f(\Om)+	\frac{N-\mu+2}{2(2N-\mu)}S_{H,L}^{\frac{2N-\mu}{N-\mu+2}}, \; \Upsilon_f^-(\Om)+	\frac{N-\mu+2}{2(2N-\mu)}S_{H,L}^{\frac{2N-\mu}{N-\mu+2}}\right)$
	   and $\{u_n\}_{n\in \mathbb{N}}$  is a $(PS)_c$ sequence of $\breve{\mathcal{J}}_f$. Then by Lemma \ref{nhlem35}, $\{t^-(u_n)u_n\}_{n\in \mathbb{N}} \in \mathcal{N}_f^-$ is a $(PS)_c$ sequence for $\mathcal{J}_f$ which on using Lemma \ref{nhlem18} gives that $\{u_n\}_{n\in \mathbb{N}}$  is compact. Moreover, from Lemma \ref{nhlem26}, $\gamma_f> \breve{\mathcal{J}}_f(h_\rho^{\e,\sigma})= 	\frac{N-\mu+2}{2(2N-\mu)}S_{H,L}^{\frac{2N-\mu}{N-\mu+2}}+o(1)$ as $\e$ sufficiently small. Using Lemma \ref{nhlem25} we have $\gamma_f$ is a critical value of $\breve{\mathcal{J}}_f$. Therefore, there exists $v_4 \in \mathcal{V}$ such that $\breve{\mathcal{J}}_f(v_4)=\gamma_f$ and $\breve{\mathcal{J}}_f^{\prime}(v_4)=0$. Thus by Lemma \ref{nhlem35}, $ u_4:= t^-(v_4) v_4  \in \mathcal{N}_f^- $ is a critical point of  $\mathcal{J}_f$ and $\mathcal{J}_f(u_4)=\gamma_f$. \QED
	\end{proof}

\noi \textbf{Proof of Theorem \ref{nhthm1} :} First note that by Lemma \ref{nhlem34}, we have all solutions of $(P_f)$ are positive in  $\Om$ and from Lemma  \ref{nhlem27}, we have $u_1 \in \mathcal{N}_f^+ \subset H_0^1(\Om) $ such that $\mathcal{J}_f(u_1)=\Upsilon_f$ whenever $0< \|f\|_{H^{-1}}<e_{00}$. By Proposition \ref{nhprop2} we have two more  critical point $u_2,u_3 \in \mathcal{N}_f^-$ of  $\mathcal{J}_f$ such that  in $\mathcal{J}_f(u_2),\mathcal{J}_f(u_3) < \Upsilon_f(\Om)+	\frac{N-\mu+2}{2(2N-\mu)}S_{H,L}^{\frac{2N-\mu}{N-\mu+2}} $.  Therefore we get three positive solutions of $(P_f)$ whenever $0< \|f\|_{H^{-1}}<e^*$. Let $e^{**}= \min\{ e^*,e_0^*\}$ then by  Proposition \ref{nhprop3},  we get $u_4\in \mathcal{N}_f^-  $  $\mathcal{J}_f(u_4)=\gamma_f$. 
\QED


\begin{thebibliography}{11}
	
\bibitem{adachifour}  S. Adachi and K. Tanaka, {\it Four positive solutions for the semilinear elliptic equation: $-\De u +u =a(x)u^p+f(x )$ in $\mathbb{R}^N$},  Calculus of Variations and Partial Differential Equations {\bf11} (2000), no. 1, 63-95.
	
	\bibitem{ambrosetti} A. Ambrosetti, {\it Critical points and nonlinear variational problems}, 
	M\`emoires de la Soci\`et\`e Math\`ematique de France {\bf49} (1992), 1-139.

	\bibitem{bahri} A. Bahri and J. M. Coron, {\it On a nonlinear elliptic equation involving the critical Sobolev exponent: the effect of the topology of the domain}, Communications on pure and applied mathematics {\bf41} (1988), no. 3,  253-294.
	
	\bibitem{benci1} V. Benci and G. Cerami, {\it  The effect of the domain topology on the number of positive solutions of nonlinear elliptic problems}, Archive for Rational Mechanics and Analysis, {\bf114} (1991), no. 1, 79-93.
	
	\bibitem{benci} V. Benci, G. Cerami and D. Passaseo, {\it  On the number of the positive solutions of some nonlinear elliptic problems}, in ``Nonlinear Analysis, A Tribute in Honour of G. Prodi",  93-109, Quaderno Scuola Normale Superiore, Pisa, 1991.
	
	\bibitem{benci2} V. Benci and G. Cerami, {\it Multiple positive solutions of some elliptic problems via the Morse theory and the domain topology},  Calculus of Variations and Partial Differential Equations {\bf2} (1994), no. 1,  29-48.
	
	
	
	\bibitem{brezis1} H. Brezis and L. Nirenberg, {\it A minimization problem with critical exponent and nonzero data}, Symmetry in Nature, Scuola Norm. Sup. Pisa {\bf 1} (1989), 129-140.
	
\bibitem{brezis2} H. Brezis and L. Nirenberg, {\it	 Remarks on finding critical points}, Communications on Pure and Applied Mathematics {\bf44} (1991), no. 8-9, 939-963.
	
	\bibitem{clapp2} M. Clapp, M. Del Pino and M. Musso, {\it  Multiple solutions for a non-homogeneous elliptic equation at the critical exponent},  Proceedings of the Royal Society of Edinburgh Section A: Mathematics {\bf 134} (2004), no. 1, 69-87. 

	\bibitem{clapp3} M. Clapp,  O. Kavian and B. Ruf, {\it Multiple solutions of nonhomogeneous elliptic equations with critical nonlinearity on symmetric domains},  Communications in Contemporary Mathematics {\bf5} (2003), no. 2, 147-169.
	 
	
	\bibitem{coron} J.M. Coron, {\it  Topologie et cas limite des injections de Sobolev}, Comptes Rendus de l'Acad\`emie des Sciences Series I Mathematics {\bf299} (1984), no. 7,  209-12.
	
	
	 
	\bibitem{daohuan} C. Dao-Min and  Z.Huan-Song, {\it Multiple positive solutions of nonhomogeneous semilinear elliptic equations in $\mathbb{R}^N$}, Proceedings of the Royal Society of Edinburgh Section A: Mathematics {\bf 126} (1996), no. 2, 443-463.
	
	\bibitem{dancer} E.N. Dancer, {\it A note on an equation with critical exponent}, Bulletin of the London Mathematical Society {\bf20} (1988), 600-602.
	
	
	
	 
	
	
	\bibitem{yangjmaa} F. Gao,  and M. Yang,  {\it On nonlocal Choquard equations with Hardy--Littlewood--Sobolev critical exponents}, Journal of Mathematical Analysis and Applications, {\bf448} (2017), no. 2, 1006-1041.
	
	\bibitem{yang} F. Gao and  M. Yang, {\it On the Brezis-Nirenberg type critical problem for nonlinear Choquard equation}, Science China Mathematics {\bf61} (2018), no. 7,  1219-1242.
	
	\bibitem{nodal} M. Ghimenti  and  J. Van Schaftingen, {\it Nodal solutions for the Choquard equation}, Journal of Functional Analysis, {\bf 271} (2016), no. 1,  107-135.
	\bibitem{systemchoq} J. Giacomoni, T. Mukherjee and K. Sreenadh, {\it Doubly nonlocal system with Hardy-Littlewood-Sobolev critical nonlinearity},  Journal of Mathematical Analysis and Applications, {\bf467} (2018), no. 1, 638-672.
	
	\bibitem{choqcoron} D. Goel,  V. D. R\u adulescu and K. Sreenadh, {\it Coron problem for nonlocal equations invloving Choquard nonlinearity}, arXiv preprint arXiv:1804.08084 (2018).
	
	\bibitem{he} H. He and J. Yang, {\it Positive solutions for critical inhomogeneous elliptic problems in non-contractible domains}, Nonlinear Analysis: Theory, Methods and  Applications, {\bf 70} (2009), no. 2, 952-973.
	
	\bibitem{hirano2} N. Hirano, { \it Multiple existence of solutions for a nonhomogeneous elliptic problem on  $\mathbb{R}^N$, } Journal of mathematical analysis and applications {\bf 336} (2007), no. 1, 506-522.
	 

	\bibitem{ehleib} E.H. Lieb, {\it Existence and uniqueness of the minimizing solution of Choquard’s
		nonlinear equation}, Studies in Applied Mathematics {\bf57} (1976/77), no. 2,  93-105. 
	\bibitem{leib} E.H. Lieb, M. Loss,  {\it Analysis} ,  graduate studies in mathematics,  American Mathematical Society, Providence , RI, 2001. 
	 
%
	

\bibitem{moroz1}	V. Moroz and J. Van Schaftingen, {\it Semi-classical states for the Choquard equation}, Calculus of  Variations and  Partial Differential Equations {\bf 52} (2015), no. 1–2, 199-235.  
\bibitem{moroz2}	V. Moroz and  J. Van Schaftingen, {\it Existence of groundstates for a class of nonlinear Choquard equations}, Transactions of  American Mathematical  Society {\bf367} (2015), 6557-6579.
	
\bibitem{moroz3} V. Moroz, J. Van Schaftingen, {\it Groundstates of nonlinear Choquard equations: Hardy-Littlewood-Sobolev critical exponent}, Communications in  Contemporary Mathematics {\bf 17} (2015), no. 5, 1550005.

	\bibitem{moroz4} V. Moroz and J. Van Schaftingen, {\it Groundstates of nonlinear Choquard equations: existence, qualitative properties and decay asymptotics}, Journal of Functional Analysis, {\bf 265} (2013), no. 2,  153-184.

	
	\bibitem{pekar} S. Pekar, {\it Untersuchung $\ddot{u}$ber die Elektronentheorie der Kristalle}, Akademie Verlag, Berlin, 1954.
%

\bibitem{rey} O. Rey, {\it A multiplicity result for a variational problem with lack of compactness},  Nonlinear Analysis: Theory, Methods and Applications {\bf 13} (1989), no. 10,  1241-1249.
 
	\bibitem{zampyang} Z. Shen, F. Gao and M. Yang, {\it Multiple solutions for nonhomogeneous Choquard equation involving Hardy-Littlewood- Sobolev critical exponent}, Zeitschrift für angewandte Mathematik und Physik {\bf 68} (2017),  no. 3, 61.
	
%
%
	\bibitem{taren}G. Tarantello,  {\it On nonhomogeneous elliptic equations involving critical Sobolev exponent,}  Annales de l'Institut Henri Poincare (C) Non Linear Analysis {\bf 9} (1992), no. 3,  281-304.
	
	\bibitem{sign} T.F. Wu, {\it Three positive solutions for Dirichlet problems involving critical Sobolev exponent and sign-changing weight}, Journal of Differential Equations, {\bf249} (2010), no. 7, 1549-1578.
 \bibitem{wu1}T.F. Wu, {\it Multiple positive solutions of non-homogeneous elliptic equations in exterior domains}, Proceedings of the Royal Society of Edinburgh Section A: Mathematics {\bf 137} (2007), no. 3,  603-624.


	
\end{thebibliography}
\end{document}